\documentclass[oneside,10pt,drafts]{amsart} 
\usepackage[b4paper]{geometry}
\usepackage[utf8]{inputenc}
\usepackage{amsfonts,amsmath,latexsym,amssymb} %
\usepackage{amsthm}   
\usepackage{mathrsfs,upref}

\theoremstyle{plain}
\newtheorem{theorem}{Theorem}[section]

\theoremstyle{plain}
\newtheorem{proposition}[theorem]{Proposition}
\newtheorem{lemma}[theorem]{Lemma}

\theoremstyle{remark}
\newtheorem{remark}[theorem]{Remark}
\newtheorem{hypothesis}[theorem]{Hypothesis}
\theoremstyle{definition}
\newtheorem{definition}[theorem]{Definition}
\newtheorem{example}[theorem]{Example}

\usepackage{graphicx,xcolor}

\usepackage{xcolor,enumerate}
\usepackage[utf8]{inputenc}

\newcommand{\dd}{\,\mathrm{d}} 
\newcommand{\ee}{\mathrm{e}}

\newcommand{\RR}{\mathbb{R}} 
\newcommand{\CC}{\mathbb{C}} 
\newcommand{\NN}{\mathbb{N}}

\newcommand{\LLL}{\mathscr{L}}
\newcommand{\Ell}{\mathrm{L}}
\newcommand{\Dir}{\text{D}}
\newcommand{\TT}{\mathbb{T}} 

\newcommand{\calT}{\mathcal{T}}
\newcommand{\calA}{\mathcal{A}}
\newcommand{\calB}{\mathcal{B}}

\newcommand{\calR}{\mathcal{R}}
\newcommand{\calS}{\mathcal{S}}

\newcommand{\vecu}{\boldsymbol{ u}}

\newcommand{\dom}{\mathrm{dom}}

\newcommand{\He}{\mathrm{H}}

\newcommand{\Ce}{\mathrm{C}}
\newcommand{\st}{:}

\newcommand{\downto}{\downarrow}
\newcommand{\trp}{\top}
\DeclareMathOperator{\Id}{I}

\newcommand{\vect}[2]{\tbinom{#1}{#2}}

\newcommand{\hpk}{\tfrac h2}

\newcommand{\tpn}{\tau}
\newcommand{\tpnh}{\tfrac\tau2}
\DeclareMathOperator{\rg}{rg}
\newcommand{\tmax}{t_{\mathrm{max}}}
\newcommand{\Lie}{{\upshape [Lie]}}
\newcommand{\Str}{{\upshape [Str]}}
\newcommand{\wgh}{{\upshape [wgh]}}

\newenvironment{iiv}{\begin{enumerate}[{\rm (i)}]}{\end{enumerate}}
\newenvironment{iivv}{\begin{enumerate}[{\rm (i')}]}{\end{enumerate}}
\newenvironment{abc}{\begin{enumerate}[{\rm (a)}]}{\end{enumerate}}

\begin{document}

\numberwithin{equation}{section}
\allowdisplaybreaks



\title[Operator splitting for abstract dynamical boundary problems]{Operator splitting for abstract Cauchy problems with dynamical boundary conditions}

\author{Petra Csom\'os}

\address{E\"otv\"os Lor\'and University,
Department of Applied Analysis and Com\-pu\-ta\-tion\-al Mathematics and MTA-ELTE Numerical Analysis and Large Networks Research Group, P\'azm\'any P\'eter s\'et\'any 1/C, 1117 Budapest, Hungary}
\email{petra.csomos@ttk.elte.hu}
\author{Matthias Ehrhardt}
\email{ehrhardt@uni-wuppertal.de}
\address{Bergische Universit\"at Wuppertal, Lehrstuhl f\"ur Angewandte Mathematik und Numerische Analysis,  Gau{\ss}strasse 20, 42119 Wuppertal, Germany}

\author{B\'alint Farkas}
\email{farkas@math.uni-wuppertal.de}
\address{Bergische Universit\"at Wuppertal, Lehrstuhl f\"ur Funktionalanalysis, 
Gau{\ss}strasse 20, 42119 Wuppertal, Germany}


\begin{abstract}
In this work we study operator splitting methods for a certain class of coupled abstract Cauchy problems, where the coupling is such that one of the sub-problems prescribes a ``boundary type'' extra condition for the other one. The theory of one-sided coupled operator matrices provides an excellent framework to study the well-posedness of such problems. We show that with this machinery even operator splitting methods can be treated conveniently and rather efficiently. We consider three specific examples:
the Lie (sequential), the Strang, and the weighted splitting, 
and prove the convergence of these methods along with error bounds under fairly general assumptions. Simple numerical examples show that the obtained theoretical bounds can be computationally realised.
\end{abstract}
\maketitle

\section{Introduction}

Operator splitting procedures provide an efficient way of solving time-dependent differential equations which describe the combined effect of several processes. 
In this case the operator describing the time-evolution is the sum of certain sub-operators corresponding to the different processes.
The main idea of operator splitting is that one solves the sub-problems corresponding to the sub-operators separately,
and constructs the solution of the original problem from the sub-solutions.

Depending on how the sub-solutions define the solution itself, one can distinguish several operator splitting procedures, 
such as sequential (proposed by Bagrinovskii and Godunov in \cite{BagGod57}), Strang (proposed by Strang and Marchuk in \cite{Strang68} and \cite{Marchuk68}), or weighted ones (see e.g.\ in Csomós et al.~\cite{CFH05}).
An application of sequential splitting, for instance, results in the subsequent solution of the sub-problems using the previously obtained sub-solution as initial condition for the next sub-problem.

Although operator splitting procedures enable the numerical treatment of complicated differential equations,
their application leads to an approximate solution which usually differs from the exact one. 
The accuracy can be increased by considering the sequence of the sub-problems on short time intervals in a cycle, which will in turn increase the computational effort. 
However, the analysis of the error, caused by the use of operator splitting, stands in the main focus of related research. 
For general overviews on splitting methods we refer the interested 
reader to the vast literature. For instance, Bj{\o}rhus analysed the consistency of sequential (Lie) splitting in an abstract framework in \cite{Bjorhus98}, Sportisse  considered the stiff case in \cite{Sportisse00}, Hansen and Ostermann also treated the abstract case in \cite{Hansen-Ostermann08}, while Bátkai et al.\ applied the splitting methods for non-autonomous evolution equations in \cite{BCsFN11}. Error bounds in the abstract setting were proved by Jahnke and Lubich in \cite{JahnkeLubich} for the Strang splitting. While Hansen and Ostermann in \cite{HO-high} have treated higher order splitting methods. A survey can be found in \cite{Geiser2011} by Geiser.

Another challenging issue is what kinds of processes of the sub-operators describe. They can e.g.\ correspond to various physical, chemical, biological, financial, etc.\ phenomena. Hundsdorfer and Verwer analysed the splitting of advection--diffusion--re\-act\-ion equations in \cite[Chapter IV]{HunVer03}, Dimov et al.\ solved air pollution transport models in \cite{Dimov-etal08}, Jacobsen et al.\ considered the Hamilton--Jacobi equations in \cite{Jacobsen-etal01}, Holden et al.\ partial differential equations with Burgers nonlinearity in \cite{HLR13}, while in \cite{Csomos-Nickel} Csomós and Nickel and in \cite{BCsF13, BCsF17} Bátkai et al.\ applied splitting methods for delay equations. Splitting methods for Schr\"odinger equations are treated, e.g., in Hochbruck et al.\ \cite{spl-schrod-damp}, Caliari et al.\ \cite{spl-magn-Schrod}. 

The sub-operators can also correspond to the change (derivative) with respect to various spatial coordinates or other variables, as Hansen and Ostermann has studied in \cite{Hansen-Ostermann08}, \cite{dim-spl-quasilin-par}; or for the case of Maxwell equations, see, e.g., Jahnke et al.\ \cite{ADI-Maxwell} or Eilinghoff and Schnaubelt \cite{ADIspl-Maxw}.
Furthermore, the sub-problems may originate from other (mathematical) properties of the problem itself, as in the present case of dynamical boundary problems.

We emphasise that the analysis and the numerical treatment of dynamical boundary problems has been attracting the attention of several researchers recently, cf.\ the work of Hipp \cite{HippDiss, Hipp19} for wave-type equations
or Knopf et al.\ \cite{Knopf20, Knopf19, KnoSig20} on the {C}ahn--{H}illiard equation
or Kovács et al.\ \cite{KLL17} and Kov\'acs, Lubich \cite{Kov17} on parabolic equations. The literature is extensive, and we mention some very recent papers by Altmann \cite{Altmann19}, Epshteyn, Xia \cite{Epshteyn20}, Fukao et al.\ \cite{Fukao17}, Langa, Pierre \cite{Langa19}, and refer to the references therein.

In the present work we focus on the \emph{abstract setting} of coupled Cauchy problems, where one of the subproblems provides an extra condition, of boundary type, to the other. We consider equations of the form:
\begin{equation}\label{eq:main0}
\begin{cases}
\begin{aligned}
 \dot u(t)&=A_m u(t)&&\text{ for } t\ge0,\quad u(0)=u_0\in E,\\
 \dot v(t)&=B v(t)&&\text{ for } t\ge0,\quad v(0)=v_0\in F, \\
	Lu(t)&=v(t) &&\text{ for } t\ge0, &
 \end{aligned}
\end{cases}
\end{equation}
where $E$ and $F$ are Banach spaces over the complex field $\CC$, $A$ and $B$ are (unbounded) linear operators on $E$ and $F$, respectively. 
The coupling of the two problems involves the unbounded linear operator $L$ acting between $E$ and $F$. 
Moreover, this coupling is of ``boundary type'', i.e., as concrete examples we have in mind problems of the following form: 
\begin{align}
	\dot{u}(t)&=\Delta_{\Omega} u(t), \qquad u(0)=u_0\in \Ell^2(\Omega),\label{eq:diff1}\\
	\dot{v}(t)&=\Delta _{\partial\Omega} v(t), \qquad v(0)=v_0\in\Ell^2(\partial\Omega), \label{eq:diff2}\\
	u(t)|_{\partial\Omega}&=v(t),\notag 
\end{align}
where $\Omega$ is a bounded domain in $\RR^d$ with sufficiently regular boundary and $A_m=\Delta_\Omega$, $B=\Delta_{\partial\Omega}$ are the (maximal) distributional Laplace and Laplace--Beltrami operators restricted to the respective $\Ell^2$-space. 
In this example $L$ denotes the trace operator (the precise ingredients will be discussed in Example~\ref{examp:beltrami} below.) 

It is a natural idea for the numerical treatment of \eqref{eq:diff1}--\eqref{eq:diff2} to apply operator splitting methods, i.e.\ 
to treat the first and second equations separately, see also in \cite{Kov17}.
The purpose of this work is to investigate such possibilities, and as a splitting strategy we propose the following steps:

\begin{enumerate}\setlength{\itemsep}{0cm}
 \item Choose a time step $\tau>0$.
	\item Solve the second equation~\eqref{eq:diff2} with the initial condition $v(0)=u_0|_{\partial\Omega}=v_0$, 
	set $v_1:=v(\tau)$.
  \item Solve the first equation~\eqref{eq:diff1}  on $[0,\tau]$  with the inhomogeneous boundary condition $u(t)|_{\partial\Omega}=v_1$ 
and the initial condition $u(0)=\widetilde u_0$. The method determines how $\widetilde u_0$ is calculated from $u_0$ (and $v_0$), and in general $\widetilde u_0$ does not need to equal $u_0$. Set $u_1=u(\tau)$.
\item The new initial condition for equation~\eqref{eq:diff2} is then $u_1|_{\partial\Omega}=v_1$.
\item Iterate this procedure for $n\in\NN$ time steps. 
\end{enumerate}

The aim of this paper is to formulate this splitting method as \emph{an operator splitting} in an abstract \emph{operator semigroup theoretic} framework and investigate its convergence properties. The method then becomes applicable for a wider class of equations than in \eqref{eq:main0}.
Our choice for the auxiliary, modified initial value $\widetilde u_0$, in Step~3 above, is motivated by this approach.
Indeed, the abstract theory will immediately yield the convergence of the method as an instance of the Lie--Trotter formula. 
However, we shall briefly touch upon other possible choices for $\widetilde u_0$ as well.

As a matter of fact our proposed methods, at a first sight, will be slightly different in that we decompose the system in not two but three sub-problems. 
This idea is nicely illustrated in the above example of diffusion: 
We separate the dynamics in the domain and assume homogeneous boundary conditions, the dynamics on the boundary, 
and as the third component the interaction between the two dynamics, i.e., how the boundary dynamics is fed into the domain.
In fact, this decomposition is responsible for the modified form $\widetilde u_0$ of the initial condition.
This approach will also have the advantage that the internal and boundary dynamics are completely separated. Hence well-established methods can be used for solving each of the subproblems.
We also note that the splitting approach here gives a way to parallelisation of the solution to the subproblems. 

This work is organised as follows. 
In Section~\ref{sec:abstractBCs} we recall the necessary operator theoretic background for this programme and in Section~\ref{sec:splittingBCs} we introduce the different splitting approaches for the dynamical boundary conditions: 
the Lie splitting, the Strang splitting and the weighted splitting. 
We also prove the convergence of these methods under fairly general assumptions. 
Section~\ref{sec:order} contains error bounds for the above mentioned splitting methods. Finally, in Section \ref{sec:num} we illustrate the proposed methods by numerical examples, and show that the analytically proved error bounds are realised computationally, too.

\section{Abstract dynamical boundary conditions}\label{sec:abstractBCs}

Before discussing splitting methods in more detail let us briefly recall a possible approach for treating such
abstract dynamical boundary value problems. The abstract treatment of boundary perturbations, i.e., techniques for altering the domain of the generator of a $C_0$-semigroup goes back to the work of Greiner \cite{Greiner}. 
Many results have been building on his theory, and our main sources for describing the abstract 
setting will be the works by Casarino, Engel, Nagel, and Nickel \cite{CENN} and Engel \cite{EngelMatrix, EngelOSC}.
In \cite{CENN} the following set of conditions were posed for treating the well-posedness of the problem~\eqref{eq:main0}.
\begin{hypothesis}\label{hyp:boundary}
	The $\CC$-vector spaces $E$ and $F$ are Banach spaces.
 \begin{iiv}
		\item The operators $A_m:\dom(A_m)\subseteq E\to E$ and $B:\dom(B)\subseteq F\to F$ are linear.
		\item The linear operator $L:\dom(A_m)\to F$ is surjective and bounded with respect to the graph norm of $A_m$ on $\dom(A_m)$.
		\item The restriction $A_0$ of $A_m$ to $\ker(L)$ generates a strongly continuous semigroup $\bigl(T_0(t)\bigr)_{t\ge0}$ on $E$.
		\item The operator $B$ generates a strongly continuous semigroup $\bigl(S(t)\bigr)_{t\ge0}$ on $F$.
  \item The operator (matrix) $\vect{A_m}{L}:\dom(A_m)\to E\times F$ is closed.
	\end{iiv}
\end{hypothesis}

\begin{remark}\label{rem:abstdir}
Consider the following conditions.
 \begin{iivv}
		\item The operator $A_m:\dom(A_m)\subseteq E\to E$ is linear.
		\item The linear operator $L:\dom(A_m)\to F$ is surjective.
		\item $L:\dom(A_m)\to F$ has a bounded right-inverse $R:F\to E$ with $\rg(R)\subseteq \ker(A_m)$.
\item 	The restriction $A_0$ of $A_m$ to $\dom(A_0):=\ker(L)$ is (boundedly) invertible (i.e., $0\in \rho(A_0)$; in general, it is sufficient to assume that the resolvent set $\rho(A_0)$ is non-empty).
	\end{iivv}
Under these assumptions $\vect{A_m}{L}:\dom(A_m)\to E\times F$ is closed. To see this, we first recall from the proof of Lemma 2.2 in \cite{CENN} that in this case
	\[
	\dom(A_m)=\dom(A_0)\oplus \ker(A_m).
	\]
	We also repeat the quick argument for this, taken from \cite{CENN}: For $x\in \dom(A_m)$ we have
	\[
	x=A_0^{-1}A_mx+(x-A_0^{-1}A_mx)\quad \text{with}\quad A_0^{-1}A_m x\in \dom(A_0),\: x-A_0^{-1}A_mx\in \ker(A_m).
	\]
 Furthermore, if $x\in\dom(A_0)\cap \ker(A_m)$, then $A_0x=A_mx=0$, and $x=0$ follows by $0\in\rho(A_0)$.
 
 \medskip\noindent 
	 Now, let $x_n\in \dom(A_m)$ be with $x_n\to x$, $A_mx_n\to y$ in $E$ and $Lx_n\to z$ in $F$ as $n\to \infty$. We need to show $x\in\dom(A_m)$, $A_mx =y$, $Lx=z$. For each $n\in \NN$ write $x_n=x^0_n+x_n^1$ with $x_n^0\in\dom(A_0)$ and $x_n^1\in\ker(A_m)$. Then $A_mx_n=A_mx_n^0+A_mx_n^1=A_mx_n^0=A_0x_n^0$, thus $x_n^0\to A_0^{-1}y$ in $E$ as $n\to \infty$. On the other hand $Lx_n=Lx_n^0+Lx_n^1=Lx_n^1\to z$ in $F$ as $n\to \infty$. It follows that
	 \[
	 RLx_n=RLx_n^1\to Rz\in \dom(A_m).
	 \]
Moreover, from $L(x_n^1-RLx_n^1)=0$	we conclude $x_n^1-RLx_n^1\in\dom(A_0)$ and $A_0(x_n^1-RLx_n^1)=A_mx_n^1-A_mRLx_n^1=0$, so that $x_n^1=RLx_n^1$ follows. This implies $x-A_0^{-1}y=Rz$, $x\in\dom(A_m)$, $A_m x=y$ and $Lx=LRz=z$. 
\end{remark}

In this paper we make the following technical assumption to simplify the things a bit.
\begin{hypothesis}\label{hyp:inv}
	The operators $A_0$ and $B$ are invertible.
\end{hypothesis}
However, let us note that for the splitting procedures this makes \emph{no theoretical difference}, 
since (for semigroup generators) one always finds sufficiently large $\lambda>0$ such that $A_0-\lambda$ and $B-\lambda$ become invertible. 
Then the numerical schemes can be applied in this rescaled situation. 

Next, we recall the following definition from \cite[Lemma 2.2]{CENN}, 
and note that under the previous assumption the following operator is bounded
\begin{align}\label{eq:D}
 D_0 &:= L|_{\ker(A_m)}^{-1}\colon F\to\ker(A_m)\subseteq E.
\end{align}
The operator $D_0$ is called the \emph{abstract Dirichlet operator}; the operator $L|_{\ker(A_m)}$ is indeed invertible, see the mentioned lemma in \cite{CENN}.
 We remark that the existence of this Dirichlet operator $D_0=:R$, a continuous right-inverse to $L$ as in Remark \ref{rem:abstdir}, is therefore equivalent to the closedness of $\vect{A_m}{L}$ (under the assumption that $0\in \rho(A_0)$).
\begin{remark}\label{rem:D0bdd}
	\begin{abc}
		\item The operator $D_0B:\dom(B)\to E$ is bounded if $\dom(B)$ is supplied with the graph-norm $\|\cdot\|_B$.
		\item We have $\rg(D_0)\cap \dom(A_0)=\{0\}$.
	\end{abc}
	\end{remark}

Following \cite{CENN} we introduce the product space $E\times F$ and the operator $\calA$ acting on it as
\begin{equation}\label{eq:calA}
 \calA:=\begin{pmatrix}A_m&0\\0& B\end{pmatrix}
  \;\text{with}\;
  \dom(\calA):=\big\{\vect{x}{y}\in\dom(A_m)\times\dom(B):Lx=y\big\}.
\end{equation}
Section~1.1 in \cite{MugnoloDiss} relates the well-posedness of \eqref{eq:main0} to the 
generation property of $\calA$, see also \cite{Mugnolo}. 

The first thing to be settled is therefore, 
whether the abstract Cauchy problem
\begin{equation*}
 \dot{\vecu}(t)=\calA u(t),\quad\text{for}\; t\ge0,
 \qquad\vecu(0)=\vecu_0=(u_0,v_0)^\trp,
\end{equation*}
is well-posed in the sense of $C_0$-semigroups, 
see \cite[Section II.6]{Engel-Nagel}.
In this case the solution satisfies $\vecu(t)=\calT(t)\vecu_0$, where $(\calT(t))_{t\ge 0}$ is the semigroup generated by $\calA$.
The problem of well-posedness is solved in \cite{CENN}. 
We briefly recall here the following results from Theorem 2.7 in \cite{CENN} and from its proof.

\begin{theorem}\label{thm:Engel}
Let the operators $\calA$, $D_0$ be as defined in \eqref{eq:calA} and \eqref{eq:D} and assume Hypotheses~\ref{hyp:boundary} and \ref{hyp:inv}. 
For $y\in\dom(B)$ define
 \begin{equation}\label{eq:Qt}
  Q(t)y=D_0S(t)y-T_0(t)D_0y - \int_0^t T_0(t-s)D_0S(s)By\dd s.
 \end{equation}
Operator $\calA$ is the generator of a $C_0$-semigroup 
if and only if for each $t\ge 0$ the operator (extends to)
 \begin{equation}\label{eq:Qtb}
  Q(t)\in\LLL(F,E) \quad\text{and}\quad \limsup_{t\downto 0}\|Q(t)\|<\infty.
 \end{equation}
 In this case the semigroup $\bigl(\calT(t)\bigr)_{t\ge0}$ generated by $\calA$ is given by 
 \begin{equation}\label{eq:calT}
  \calT(t)=\begin{pmatrix}T_0(t)& Q(t)\\0 & S(t)\end{pmatrix}.
 \end{equation}
\end{theorem}

The next condition will be important throughout the paper.
\begin{hypothesis}\label{hyp:sec}
	The operator $A_0$ generates a boun\-ded analytic semigroup (see \cite{Lunardi, Haase} for details about analytic semigroups). 
	\end{hypothesis}
If Hypotheses~\ref{hyp:boundary}, \ref{hyp:inv} and \ref{hyp:sec} are fulfilled and also $B$ is a generator of an analytic semigroup, then Theorem~\ref{thm:Engel} applies and assures that the semigroup $(\calT(t))_{t\ge0}$ generated by $\calA$ is analytic too, see \cite[Corollary 2.8]{CENN}.


The motivating example from the introduction is discussed in \cite[Section 3]{CENN} in detail. We recall here the ingredients, to illustrate that our proposed methods will be applicable also for this equation.

\begin{example}[Laplace and Laplace--Beltrami operators]
	\label{examp:beltrami}
Let $\Omega$ be a bounded domain in $\RR^d$ with boundary $\partial\Omega$ of class $\Ce^2$.
	\begin{itemize}
		\item $E:=\Ell^2(\Omega)$, $F:=\Ell^2(\partial\Omega)$ are the $\Ell^2$-spaces with respect to the Lebesgue and the surface measure, respectively.
		\item $\Delta_\Omega$ and $\Delta_{\partial\Omega}$ are the (maximal) distributional Laplace and Laplace--Beltrami operators, respectively.
		\item $A_m:=\Delta_\Omega$ with domain 
\begin{equation*}
		\dom(A_m):=\{f\st f\in \He^{1/2}(\Omega)\text{ with }\Delta_\Omega f\in \Ell^2(\Omega)\}.
		\end{equation*}
		\item $Lf=f|_{\partial \Omega}$ the trace of $f\in \dom(A_m)$ on $\partial\Omega$.
	\item $B=\Delta_{\partial\Omega}$ with domain 
	\begin{equation*}
	\dom(B)=\{g\st g\in \Ell^2(\partial\Omega)\text{ with }\Delta_{\partial\Omega} g\in \Ell^2(\Omega\}.
	\end{equation*}
\end{itemize}
 Hypotheses \ref{hyp:boundary}, \ref{hyp:inv}, \ref{hyp:sec} are satisfied for these choices. In particular, $\calA$ generates an analytic semigroup on $E\times F$, see \cite[Section 3]{CENN}. We also have the following:
 \begin{itemize}
	 \item The Dirichlet operator $D_0:\Ell^2(\partial\Omega)\to 
 \He^{1/2}(\Omega)$ assigns to a prescribed boundary value $g$ a function $f$ with $f|_{\partial \Omega}=g$ (in the sense of traces) and $\Delta_{\Omega}f=0$.
 \item $A_0= \Delta_{\Dir}$ is the Laplace operator with (homogeneous) Dirichlet boundary condition, generating the Dirichlet heat semigroup $(T_0(t))_{t\geq 0}$ on $\Ell^2(\Omega)$.
 \item The semigroup $(S(t))_{t\geq 0}$ is the heat semigroup on $\Ell^2(\partial\Omega)$.
\end{itemize}
	\end{example}

\begin{example}[Bounded Lipschitz domains]\label{examp:Lip}
In this example we indicate that one can relax the smoothness condition on the boundary of the domain from Example \ref{examp:beltrami}.
Let $\Omega\subseteq \RR^d$ be a bounded domain with Lipschitz boundary $\partial\Omega$.
\begin{abc} 
\item Consider the following operators:
	\begin{itemize}
		\item $A_m=\Delta_\Omega$ with domain 
	\begin{equation*}
		\dom(A_m):=\{f\st f\in \He^{1/2}(\Omega)\text{ with }\Delta_\Omega f\in \Ell^2(\Omega)\}.
		\end{equation*}
		\item $Lf=f|_{\partial \Omega}$ the trace of $f\in \dom(A_m)$ on $\partial\Omega$ (see, e.g., \cite[pp.{} 89--106]{McL00}).
\end{itemize}
Then $L$ is surjective and actually has a bounded right-inverse
		\[
		R:\Ell^2(\partial \Omega)\to \ker(A_m),
		\] where $\dom(A_m)$ is endowed with the norm $u\mapsto \|u\|_{\He^{1/2}}+\|\Delta u\|_2$, see Theorem 3.6 (i) in \cite{BeGeMi20} for precisely this statement (or \cite[Lemma 3.1, Theorem 5.3]{GeMit11}, \cite{GeMit08}, \cite[Theorem 3.37]{McL00}).
The restriction $A_0$ of $A_m$ to 
\[
\ker(L)=\{f\st \He^{1/2}(\Omega),\:\Delta_\Omega f\in \Ell^2(\Omega),\: Lf=0\}
\]		
is strictly positive, self-adjoint, in particular $A_0$ is invertible and generates a bound\-ed analytic semigroup, \cite[Theorem 5.1]{GeMit11}, \cite[Theorem 2.11]{GeMit08} (see also \cite[Theorem 3.6 (v)]{BeGeMi20}). By invoking Remark \ref{rem:abstdir} we obtain that $\vect{A_m}{L}:\dom(A_m)\to \Ell^2(\Omega)\times \Ell^2(\partial\Omega)$ is closed, and altogether that $A_m$ and $L$ satisfy the relevant conditions from Hypothesis \ref{hyp:boundary}. 
\item One can also consider the Laplace--Beltrami operator $B:=\Delta_{\partial\Omega}$ on $\Ell^2(\partial \Omega)$, which (with an appropriate domain) is also a strictly positive, self-adjoint operator, see \cite[Theorem 2.5]{GeMit11} or \cite{GeMiMiMi} for details.
\end{abc}
Summing up, we see that the abstract framework of \cite{CENN}, hence of this paper, covers also some interesting cases of dynamical boundary value problems on bounded Lipschitz domains.
\end{example}

The decisive tool, based on the theory of coupled operator matrices \cite{EngelMatrix, EngelOSC}, 
is to bring the \emph{formally diagonal} operator $\calA$ with a \emph{non-diagonal domain} 
into an upper triangular form with the state space transformations
 \begin{equation*}
 \calR_0=\begin{pmatrix}I&-D_0\\0& I\end{pmatrix},\qquad \calR_0^{-1}=\begin{pmatrix}I&D_0\\0& I\end{pmatrix}.
 \end{equation*}
Accordingly, we obtain the following representation: 
\begin{equation}\label{eq:long}
	\calA = \mathcal R_0^{-1}\calA_0\mathcal R_0,
\end{equation}
where 
\begin{equation*}
	\calA_0 = \begin{pmatrix}A_0&-D_0B\\0&B\end{pmatrix}\quad \text{with}\quad\dom(\calA_0)=\dom(A_0)\times \dom(B),
\end{equation*}
see \cite[Lemma 2.6 and the proof of Corollary 2.8]{CENN}. 
	
\section{Operator splitting methods for dynamical boundary conditions problems}\label{sec:splittingBCs}
Since the form of the semigroup $(\calT(t))_{t\ge 0}$ can be rarely determined in practice, our aim is to determine an approximation to it,
and denote at time $t=k\tau$ the approximation of $\vecu(k\tau)$ 
by $\vecu_k(\tau)$ for all $k\in\NN$. 
The natural requirement is that the approximate value should converge 
to the exact one when refining the temporal resolution (letting $\tau\to 0$). We recall the following definition from \cite{Lax56} due to Lax and Richtmyer.

\begin{definition}[Convergence]
The approximation $\vecu_k$ is called convergent to the solution $\vecu$ of problem~\eqref{eq:main0} on $[0,\tmax]$ (for  given $\tmax>0$) if
 $\vecu(t)=\lim\limits_{n\to\infty}\vecu_n(\tfrac tn)$ holds uniformly for all $t\in[0,\tmax]$.
\end{definition}

Starting from the representation~\eqref{eq:long}, 
we construct approximations of the form
\begin{equation}\label{eq:nummethod}
	\vecu_k(\tau):=\mathcal R_0^{-1}\TT(\tau)^k\mathcal R_0\tbinom{u_0}{v_0},
\end{equation}
where the operator $\TT(\tau)\colon E\times\dom(B)\to E\times\dom(B)$,
$\tau\ge0$ 
describes the actual numerical method, and $\vecu(0)=\vecu_0=(u_0,v_0)^\top$. 
In order to specify the operator $\TT(\tau)$, we remark that the operator $\calA_0$ can be written as the sum
\begin{equation*}
	\calA_0=:\calA_1+\calA_2+\calA_3,
\end{equation*}
where
\begin{equation*}
 \calA_1 = \begin{pmatrix}A_0&0\\0&0\end{pmatrix},
	\quad 
	\calA_2 = \begin{pmatrix}0&-D_0B\\0&0\end{pmatrix}, 
	\quad 
	\calA_3 = \begin{pmatrix}0&0\\0&B\end{pmatrix}, 
	\end{equation*}
	with
	\begin{equation*}
	\dom(\calA_1)=\dom(A_0)\times F, \;
	\dom(\calA_2)=E\times\dom(B), \;
	\dom(\calA_3)=E\times\dom(B).
\end{equation*}
We point out that $\calA_1$ and $\calA_3$ commute (in the sense of resolvents).
From Hypothesis~\ref{hyp:boundary} and Remark~\ref{rem:D0bdd} we immediately obtain the following
proposition.

\begin{proposition} 
		The operator semigroups $(\calT_i(t))_{t\ge 0}$, $i=1,2,3$ given by
\begin{equation*}
	\calT_1(t) = \begin{pmatrix}T_0(t)&0\\0&I\end{pmatrix}, \quad \calT_2(t) = \begin{pmatrix}I&-tD_0B\\0&I\end{pmatrix}, \quad \calT_3(t) = \begin{pmatrix}I&0\\0&S(t)\end{pmatrix}
\end{equation*}
are strongly continuous on $E\times \dom(B)$ with generator 
\begin{equation*}
\text{$\calA_1|_{E\times \dom(B)}$, $\calA_2$ and $\calA_3|_{E\times \dom(B)}$, respectively.}
\end{equation*}
Here we consider the parts of the respective operators in the space $E\times \dom(B)$.
The semigroups $(\calT_1(t))_{t\ge 0}$ and $(\calT_3(t))_{t\ge 0}$ are even strongly continuous on $E\times F$. 
Their generators are $\calA_1$ and $\calA_3$, respectively.
\end{proposition}

In this work we focus on methods~\eqref{eq:nummethod} with the following choices for the operator $\TT(\tau)$:
\begin{align}
 \TT^\text{\Lie}(\tau)&:=\calT_1(\tau)\calT_2(\tau)\calT_3(\tau)	\label{eq:lie} \qquad \intertext{for the \emph{Lie (or sequential) splitting};}
	\TT^\text{\Str}(\tau)&:=\calT_1(\tfrac\tau 2)\calT_3(\tfrac\tau 2)\calT_2(\tau)\calT_3(\tfrac\tau 2)\calT_1(\tfrac\tau 2) \label{eq:str} 
	\qquad \intertext{for the \emph{Strang (or symmetrical) splitting};}
 \TT^\text{\wgh}(\tau)&:=\Theta\calT_1(\tau)\calT_2(\tau)\calT_3(\tau)+(1-\Theta)\calT_3(\tau)\calT_2(\tau)\calT_1(\tau) 	\label{eq:wgh}
\end{align}
for the \emph{weighted splitting}, where the parameter $\Theta\in[0,1]$ is fixed.
 We note that the case $\Theta=1$ corresponds to the Lie splitting, 
 while $\Theta=0$ gives the Lie splitting in the reverse order.
Computing the composition of the operators leads to the common form
\begin{equation}\label{eq:TT}
	\TT(\tau)=\begin{pmatrix}T_0(\tau)&V(\tau)\\0&S(\tau)\end{pmatrix}
\end{equation}
with the operators
\begin{align}
	 \text{Lie splitting:}\quad V^\text{\Lie}(\tau) &= -\tau T_0(\tau)D_0BS(\tau), \label{eq:Vlie}\\
	 \text{Strang splitting:}\quad V^\text{\Str}(\tau) &= -\tau T_0(\tfrac\tau 2)D_0BS(\tfrac\tau 2),\label{eq:Vstr}\\
	 \text{weighted splitting:}\quad V^\text{\wgh}(\tau) &= -\tau (\Theta T_0(\tau)D_0BS(\tau)+(1-\Theta)D_0B) 	\label{eq:Vwgh}
\end{align}
for all $\tau >0$. The approximation \eqref{eq:nummethod} requires the powers of the operator $\TT(\tau)$ to be computed next.

\begin{proposition}\label{prop:Tk}
For the operator family $\TT(\tau)\colon E\times\dom(B)\to E\times \dom(B)$, $\tau >0$, from \eqref{eq:TT} we have the identity
	\begin{align}
\nonumber		\TT(\tau)^k &= \begin{pmatrix}T_0(k\tau)&V_k(\tau)\\0&S(k\tau)\end{pmatrix} 
\intertext{with} 
\label{eq:Vk}	V_k(\tau) &= \sum\limits_{j=0}^{k-1}T_0\bigl((k-1-j)\tau\bigr)V(\tau)S(j\tau).
	\end{align}
\end{proposition}
\begin{proof}
We show the assertion by induction. 
For $k=1$ we have formula~\eqref{eq:TT} with $V_1(\tau)=V(\tau)$. 
If the assertion is valid for some $k\ge 1$, then
\begin{equation*}
\begin{split}
	\TT(\tau)^{k+1} &=  
	 \begin{pmatrix}T_0(k\tau)&V_k(\tau)\\0&S(k\tau)\end{pmatrix}
	 \begin{pmatrix}T_0(\tau)&V(\tau)\\0&S(\tau)\end{pmatrix}
	=\begin{pmatrix}T_0\bigl((k+1)\tau\bigr)&V_{k+1}(\tau)\\
	    0&S\bigl((k+1)\tau\bigr)\end{pmatrix}
	\intertext{holds with}
	V_{k+1}(\tau)&=T_0(k\tau)V(\tau)+V_k(\tau)S(\tau)\\& = T_0(k\tau)V(\tau)+\sum_{j=0}^{k-1}T_0\bigl((k-1-j)\tau\bigr)V(\tau)S(j\tau)S(\tau) \\
	&= T_0(k\tau)V(\tau)+\sum_{j=1}^k T_0\bigl((k-j)\tau\bigr)V(\tau)S\bigl((j-1)\tau\bigr)S(\tau)\\
	&= \sum\limits_{j=0}^k T_0\bigl((k-j)\tau\bigr)V(\tau)S(j\tau).
\end{split}
\end{equation*}
This proves the assertion for all $k\in\NN$ by induction.
\end{proof}

The convergence of the approximation relies on the following result. 
\begin{proposition}\label{prop:Vn}
Under Hypotheses \ref{hyp:boundary}, \ref{hyp:inv}, \eqref{eq:Qtb} and with the notation in~\eqref{eq:TT}, the approximation~\eqref{eq:nummethod} is convergent for $y\in\dom(B)$ if the condition
	\begin{equation}\label{eq:Vlim}
		\lim\limits_{n\to\infty}V_n(\tfrac tn)y
		= -\int_0^t T_0(t-s)D_0S(s)By\,\dd s
	\end{equation}
	holds uniformly for $t$ in compact intervals.
\end{proposition}
\begin{proof}
From Proposition~\ref{prop:Tk}, the approximation has the form
\begin{equation}\label{eq:nummethod2}
\begin{aligned}
	\vecu_k(\tau)&=\begin{pmatrix}I&D_0\\0& I\end{pmatrix}\begin{pmatrix}T_0(k\tau)&V_k(\tau)\\0&S(k\tau)\end{pmatrix}\begin{pmatrix}I&-D_0\\0& I\end{pmatrix}\vecu_0 \\
	&= \begin{pmatrix}T_0(k\tau)&V_k(\tau)-T_0(k\tau)D_0+D_0S(k\tau)\\0&S(k\tau)\end{pmatrix}\vecu_0.
	\end{aligned}
\end{equation}
By comparing with formula \eqref{eq:calT} and using the relation \eqref{eq:Qt}, condition \eqref{eq:Vlim} implies the assertion.
\end{proof}

The convergence of the Riemann sums implies our next result concerning the approximation of the convolution in \eqref{eq:Vlim}.
\begin{lemma}\label{lem:Cn}
 Let $\tmax\ge 0$, let $f\colon[0,\tmax]\to\LLL(F,E)$ be strongly continuous, and let $g\colon [0,\tmax]\to F$ be continuous. For each $n\in\NN$ and $t\in [0,\tmax]$ define the following expressions
 \begin{align*}
  C_n^{[1]}(t) &:= \tfrac tn\sum_{j=0}^{n-1}f\bigl((n-j)\tfrac tn\bigr)g(j\tfrac tn), \\
  C_n^{[2]}(t) &:= \tfrac tn\sum_{j=0}^{n-1}f\bigl((n-j-\tfrac 12)\tfrac tn\bigr)\,g\bigl((j+\tfrac 12)\tfrac tn\bigr).
 \end{align*}
Then for $j=1,2$ we have that 
 \begin{equation*}
  \lim\limits_{n\to\infty}C_n^{[j]}(t) = \int_0^t f(t-s)\,g(s)\dd s
 \end{equation*}
 holds uniformly for $t\in[0,\tmax]$.
\end{lemma}

We can now state the main result of this section concerning convergent approximations of the solution to problem \eqref{eq:main0}.
\begin{proposition}
Under Hypotheses \ref{hyp:boundary}, \ref{hyp:inv}, and \eqref{eq:Qtb} 
the approximations defined in \eqref{eq:lie}, \eqref{eq:str}, and \eqref{eq:wgh} are convergent for all $\vecu_0\in E\times \dom(B)$.
\end{proposition}
\begin{proof}
It suffices to prove that condition \eqref{eq:Vlim} holds for the operators $V(\tau)$ defined in \eqref{eq:Vlie}, \eqref{eq:Vstr} and \eqref{eq:Vwgh}.
By Proposition~\ref{prop:Tk}, we have the following identity for the Lie splitting:
\begin{align*}
	V_k^\text{\Lie}(\tau)y &= -\tau\sum\limits_{j=0}^{k-1} T_0\bigl((k-1-j)\tau\bigr)T_0(\tau)D_0BS(\tau)S(j\tau)y \\
	&= -\tau\sum\limits_{j=0}^{k-1} T_0\bigl((k-j)\tau\bigr)D_0BS\bigl((j+1)\tau\bigr)y,
\end{align*}
for the Strang splitting:
\begin{align}
	\label{eq:StrV}V_k^\text{\Str}(\tau)y &= -\tau \sum\limits_{j=0}^{k-1} T_0\bigl((k-1-j)\tau\bigr)T_0(\tfrac\tau 2)D_0BS(\tfrac\tau 2)S(j\tau)y \\
	\notag &= -\tau \sum\limits_{j=0}^{k-1} T_0\bigl((k-j-\tfrac 12)\tau\bigr)D_0BS\bigl((j+\tfrac 12)\tau\bigl)y,
\end{align}
and for the weighted splitting:
\begin{align*}
 V_k^\text{\wgh}(\tau)y &= -\tau\sum\limits_{j=0}^{k-1} T_0\bigl((k-1-j)\tau\bigr)
 \big(\Theta T_0(\tau)D_0BS(\tau)+(1-\Theta)D_0B\big) S(j\tau)y \\
 &= -\Theta\tau\sum\limits_{j=0}^{k-1}T_0\bigl((k-j)\tau\bigr)D_0BS\bigl((j+1)\tau\bigr)y \\
 &\quad-(1-\Theta)\tau\sum\limits_{j=0}^{k-1}T_0\bigl((k-j-1)\tau\bigr)D_0BS(j\tau)y
\end{align*}
for all $y\in\dom(B)$, $\tau >0$, and $\Theta\in[0,1]$. 
Since $B$ and the semigroup operators $S(t)$ commute on $\dom(B)$, we have
\begin{align*}
	V_n^\text{\Lie}(\tfrac tn)y &= -\tau \sum\limits_{j=0}^{n-1}T_0((n-j)\tfrac tn)D_0S((j+1)\tfrac tn)By, \\
	V_n^\text{\Str}(\tfrac tn)y &= -\tau \sum\limits_{j=0}^{n-1}T_0((n-j-\tfrac 12)\tfrac tn)D_0S((j+\tfrac 12)\tfrac tn)By, \\
	V_n^\text{\wgh}(\tfrac tn)y &= -\Theta\tau\sum\limits_{j=0}^{n-1}T_0((n-j)\tfrac tn)D_0S((j+1)\tfrac tn)By \\
	\nonumber &\quad -(1-\Theta)\tau\sum\limits_{j=0}^{n-1}T_0((n-j-1)\tfrac tn)D_0S(j\tfrac tn)By.
\end{align*}
Now, Lemma~\ref{lem:Cn} yields the convergence to the convolution in \eqref{eq:Vlim} for each of these cases. 
\end{proof}

\begin{remark}
The stability of splitting methods for triangular operator matrices has been studied in \cite{BCsEF}. If we write
\begin{equation*}
 \calA_0 = \begin{pmatrix}A_0&0\\0&B\end{pmatrix}+\begin{pmatrix}0&-D_0B\\0&0\end{pmatrix}=\calB+\calA_2,
\end{equation*}
then $\calB$ with $\dom(\calB)=E\times \dom(B^2)$ generates the strongly continuous semigroup
\begin{equation*}
	\calS(t) = \begin{pmatrix}T_0(t)&0\\0&S(t)\end{pmatrix}
	\end{equation*}
on $E\times \dom(B)$. Since $\calA_2$ is bounded on this space, by \cite[Prop.{} 2.4]{BCsEF} we obtain that for some $M\ge 0$ and $\omega \in \RR$
\begin{equation*}
 \bigl\|\bigl(\calS(\tfrac tn)\calT_2(\tfrac tn)\bigr)^n\bigr\|_{\LLL(E\times \dom(B))}\le M\ee ^{\omega }\quad \text{for every $t\ge 0$.}
\end{equation*}
Thus we immediately obtain the convergence of the corresponding Lie splitting procedure on $E\times \dom (B)$ by the Lie--Trotter product formula, see \cite[Section III.5]{Engel-Nagel}, or \cite{Trotter59}. 
As a matter of fact, in this way we obtain also the generator property of $\calA_0$ on $E\times \dom (B)$ directly, without recurring to \cite{CENN}. 
\end{remark}

\begin{remark}\label{rem:V*rem}
Let us comment on the relation between the previously proposed Lie splitting and the one from the introduction.
Given $u_0\in \He^{1/2}(\Omega)$ such that $v_0=u_0|_{\partial\Omega}$ belongs to $\dom (B)=\He^2(\partial \Omega)$, we have 
that the Lie splitting corresponds to the choices $v_1=S(\tau)v_0\in \dom(B)$ and
\begin{align*}
	\widetilde u_0&=u_0-D_0 v_0+D_0 v_1-\tau D_0 Bv_1=u_0+D_0\Bigl(\int_0^\tau S(r)Bv_0\, \dd r -\tau D_0 Bv_1\Bigr)\\
	&= u_0+D_0\int_0^\tau \bigl( S(r)-S(\tau)\bigr)Bv_0\,\dd r.
\end{align*}
If $v_0\in\dom(B^2)$, 
we obtain $\widetilde{u}_0=u_0+\mathrm{O}(\tau^2)$, where $\mathrm{O}(\tau^2)$ denotes a term with norm less than or equal to $\tau^2 C \|B^2v_0\|$. 

It can be proved that if a method (more precisely the choice of $\widetilde{u}_0$) satisfies $\widetilde{u}_0=u_0+\mathrm{O}(\tau^2)$, then the corresponding splitting method (e.g.\ the one in the introduction with $\widetilde u_0=u_0$) is convergent. 
In addition, its convergence order is the same as for the Lie splitting, cf.{} the next section.
\end{remark}

\section{Order of convergence}\label{sec:convergenceLie}
\label{sec:order}
In this section we will investigate the order of convergence of the proposed splitting schemes.
We begin with recalling a standard definition, see, e.g., \cite{Atkin09}.

\begin{definition}[Order of convergence]
The approximation $\vecu_n$ to $\vecu$ is called \textit{convergent of order} $p>0$ on $[0,\tmax]$ (for some fixed $\tmax>0$) if there exists a constant $C\ge0$ such that $\|\vecu(t)-\vecu_n(\tfrac tn)\|\le C n^{-p}$ for every $t\in[0,\tmax]$ and $n\in\NN\setminus\{0\}$. 
\end{definition}
 
 The rest of this paper is devoted to the investigation of such estimates for the approximations given in \eqref{eq:nummethod}. 
\begin{remark}
Jahnke and Lubich \cite{JahnkeLubich} studied the convergence order of the Strang splitting for
generators of bounded analytic semigroups under certain commutator conditions (for the Lie splitting, see \cite[Chapter 10]{ISEM2012}). If we split
\begin{equation*}
 \calA_0 = \begin{pmatrix}A_0&0\\0&B\end{pmatrix}+\begin{pmatrix}0&-D_0B\\0&0\end{pmatrix}=\calB+\calA_2,
\end{equation*}
and assume that $A_0$, $B$ are generators of bounded analytic semigroups, then in order to apply their result we need to calculate the commutator $[\calB,\calA_2]$.
We have
\begin{align*} 
 [\calB,\calA_2]&=\begin{pmatrix}A_0&0\\0&B\end{pmatrix}\begin{pmatrix}0&-D_0B\\0&0\end{pmatrix}-\begin{pmatrix}0&-D_0B\\0&0\end{pmatrix}\begin{pmatrix}A_0&0\\0&B\end{pmatrix}\\
 &=\begin{pmatrix}0&-A_0D_0B\\0&0\end{pmatrix}-\begin{pmatrix}0&D_0B^2\\0&0\end{pmatrix}=-\begin{pmatrix}0&D_0B^2\\0&0\end{pmatrix}
\end{align*}
with the domain 
\begin{equation*}
 \dom([\calB,\calA_2])=\dom(A_0)\times \{0\},
\end{equation*} 
by Remark~\ref{rem:D0bdd}. 
This renders the direct application of the Jahnke--Lubich result impossible. 
Moreover, in contrast to \cite{JahnkeLubich} we do not need to require that the operator $B$ is also an analytic generator, only the well-posedness of~\eqref{eq:main0} (or equivalently \eqref{eq:Qtb}). 
The price to be paid for this simplification is the requirement of 
increased regularity conditions for the initial value.
\end{remark}

Before proceeding to the error estimates we start 
with an important observation, whose proof is a small modification 
of the one of Lemma~\ref{prop:Vn}, cf.\ \eqref{eq:nummethod2}.
\begin{proposition}
	Let $V(\tau)$ be as in \eqref{eq:TT} and let 
	$D\subseteq F$ be a subspace with a given norm $\|\cdot\|_D$. Let $r\ge 0$, let $\tmax>0$ and $C\ge 0$ such that for every $y \in D$ and for every $t\in [0,\tmax]$
	\begin{equation}\label{eq:orderVn}
\Bigl\|V_n(\tfrac tn)y +\int_{0}^{t} T_0(t-s)D_0 S(s)By\dd s\Bigr\|\le \frac{Ct^r\log(n)}{n^r}\|y\|_D.
	\end{equation}
 Then 
	\begin{equation*}
	\Bigl\|\calR_0^{-1}\TT^n(\tfrac tn)\calR_0\tbinom{x}{y}-\calT(t)\tbinom{x}{y}\Bigr\|\le \frac{Ct^r\log(n)}{n^r}\|y\|_D
	\end{equation*}
	for every $x\in E$, $y\in D$ and $t\in [0,\tmax]$.
	In particular, the approximation $\vecu_k$ defined in \eqref{eq:nummethod} is convergent of order $p$ for any $p\in (0,r)$ and every initial value $\vecu_0\in E\times D$.
\end{proposition}
From now on we will focus on the error estimates concerning the approximation $V_n(\tfrac tn)$, where the corresponding $V$ is either given in \eqref{eq:Vlie}, or \eqref{eq:Vstr} or \eqref{eq:Vwgh} (but note that many other choices for $V$ are possible, cf.\ Remark~\ref{rem:V*rem}.)

\begin{lemma}[Local error of splittings I]\label{lem:loclieA}
	Let $A_0$ and $B$ be the generator of the strongly continuous semigroups $(T_0(t))_{t\ge0}$ and $(S(t))_{t\ge0}$, respectively, and 
	suppose Hypotheses \ref{hyp:boundary}, \ref{hyp:inv}, \eqref{eq:Qtb}, \ref{hyp:sec}.
For every $\tmax>0$ there is $C\ge 0$ such that for every $h\in[0,\tmax]$, 
for every $s_0,s_1\in [0,h]$, and for every $y\in \dom(B^2)$ we have
\begin{align*}
  \Bigl\|\int_{0}^{h} T_0(h-s)A_0^{-1}D_0 S(s)By\dd s-hT_0(h-s_0)A_0^{-1}D_0 S(s_1)By\Bigr\|
  \le C h^{2}(\|By\|+\|B^2 y\|).
\end{align*} 
\end{lemma}
\begin{proof}
For each $y\in\dom(B^2)$ we can write
\begin{align*}
 &\int_{0}^{h} T_0(h-s)A_0^{-1}D_0 S(s)By\dd s-hT_0(h-s_0)A_0^{-1}D_0 S(s_1)By\\
 &=\int_{0}^{h} \bigl( T_0(h-s)A_0^{-1}D_0 S(s)By-T_0(h-s_0)A_0^{-1}D_0 S(s_1)By\bigr)\dd s\\
 &= \int_{0}^{h} (T_0(h-s)-T_0(h-s_0))A_0^{-1}D_0 S(s)By\dd s\\
 &\quad+ \int_{0}^{h} T_0(h-s_0)A_0^{-1}D_0 (S(s)- S(s_1))By\dd s=I_1+I_2,
\end{align*}
where $I_1$, $I_2$ denote the occurring integrals on the right-hand side in the order of appearance.
%
The first term $I_1$ can be estimated as
\begin{align*}
 \|I_1\|&=\Bigl\|\int_{0}^{h} (T_0(h-s)-T_0(h-s_0))A_0^{-1}D_0S(s)By\dd s\Bigr\|\\
 &\le \int_{0}^{h} \|(T_0(h-s)-T_0(h-s_0))A_0^{-1}\|\cdot \| D_0\| \cdot\|S(s)By\|\dd s\\
 &\le C_1\int_{0}^{h} |s-s_0| \dd s\cdot \|By\|= C_2 h^2 \|By\|.
\end{align*}
For the second term $I_2$ we obtain the estimate:
\begin{align*} 
 \|I_2\|&= \Bigl\|\int_{0}^{h} T_0(h-s_0)A_0^{-1}D_0 (S(s)- S(s_1))By\dd s\Bigr\|\\
 & \le C_3\int_{0}^{h} \|(S(s)- S(s_1))By\|\dd s
 = C_3\int_{0}^{h}\Bigl\|\int_{s_1}^s S(r)B^2y\dd r\Bigr\|\dd s \\
 & \le C_4 h^2\|B^2 y\|.
\end{align*}
Putting these estimates together finishes the proof of the lemma.
\end{proof}

The validity of the following condition makes it possible to prove convergence rates for the other types of splittings.
\begin{hypothesis}\label{hyp:fracpow}
(We suppose as in Hypothesis \ref{hyp:sec} that $A_0$ generates a boun\-ded analytic semigroup.)
The number $\gamma\in[0,1]$ is such that $\rg(D_0)\subseteq \dom((-A_0)^\gamma)$.
\end{hypothesis}
We refer to \cite[Chapter 3]{Haase}, \cite[Chapter 4]{Lunardi}, \cite[Chapter II.5]{Engel-Nagel} or \cite[Chapter 9]{ISEM2012} for details concerning fractional powers of sectorial operators.
In particular, at this point it is important to recall the following result.
\begin{remark}\label{rem:fracpowbdd}
	If $\alpha\in[0,1]$ and $A_0$ is the generator of the bounded analytic semigroup $(T_0(t))_{t\ge0}$, then 
	\begin{equation*}
   	\sup_{t>0} \|t^\alpha(-A_0)^{\alpha}T_0(t)\|<\infty, 
	\end{equation*}
and	\[
	\sup_{t>0} \|t^{-\alpha}(T_0(t)-I)(-A_0)^{-\alpha}\|<\infty.
	\]
	\end{remark}
\begin{remark}\label{rem:fracpowbdd2}
\begin{abc}
\item For $\gamma=0$ the condition in Hypothesis \ref{hyp:fracpow} is always trivially satisfied, and this choice will suffice for the Lie splitting. The requirement $\gamma>0$ is only needed for the cases of the Strang and the weighted splittings.
\item Hypothesis~\ref{hyp:fracpow} is fulfilled in the setting of Example~\ref{examp:beltrami} for the Dirichlet Laplace operator $\Delta_{\Dir}$ with $\gamma\in [0,1/4)$. 
Indeed, we have $\rg (D_0)\subseteq\He^{1/2}(\Omega)$. 
For $\gamma\in [0,1/4)$ we have by \cite[Theorem 11.1]{LMa} that
 \begin{equation*}
   \He^{2\gamma}(\Omega)=\He^{2\gamma}_0(\Omega),
\end{equation*}
and then by complex interpolation, \cite[Theorem 11.6]{LMa}, we can write
\begin{equation*}
\bigl[\He_0^2(\Omega),\Ell^2(\Omega)\bigl]_{\gamma}
  =\He^{2\gamma}(\Omega).
\end{equation*}
Moreover, since 
 \begin{equation*}
\He_0^2(\Omega)\subseteq \He^1_0(\Omega)\cap \He^2(\Omega)=\dom( \Delta_{\Dir})
\end{equation*}
with continuous inclusion, we obtain (see, e.g., \cite[Chapter 4]{Lunardi}) that
 \begin{equation*}
 \He^{2\gamma}(\Omega)=
 \bigl[\He_0^2(\Omega), \Ell^2(\Omega)\bigr]_{\gamma}\subseteq \bigl[\dom(\Delta_{\Dir}),\Ell^2(\Omega)\bigr]_\gamma \subseteq\dom\bigl((- \Delta_{\Dir})^{\gamma}\bigr).
\end{equation*}
Finally, this yields
 \begin{equation*}
  \rg(D_0)\subseteq \He^{1/2}(\Omega)\subseteq \He^{2\gamma}(\Omega)\subseteq \dom\bigl((- \Delta_{\Dir})^{\gamma}\bigr).
\end{equation*}

\item It is important to note that if for some $\gamma\ge0$ Hypothesis~\ref{hyp:fracpow} is satisfied, then $(-A_0)^\gamma D_0:F\to E$ is a closed, and hence bounded, linear operator.
\end{abc}
\end{remark}
\begin{lemma}[Local error of splittings II]\label{lem:D0fracpow}
	Let $A_0$ and $B$ be the generator of the strongly continuous semigroups $(T_0(t))_{t\ge0}$ and $(S(t))_{t\ge0}$, respectively.	Suppose Hypotheses \ref{hyp:boundary}, \ref{hyp:inv}, \eqref{eq:Qtb}, \ref{hyp:sec} and also
 \ref{hyp:fracpow}, i.e., that $\rg(D_0)\subseteq\dom((-A_0)^\gamma)$ for some $\gamma\in[0,1]$.
For every $\tmax>0$ there is $C\ge0$ such that for every $h\in[0,\tmax]$, 
for every $s_0,s_1\in[0,h]$ and for every $y\in\dom(B^2)$ we have
\begin{equation*}
 \Bigl\|\int_{0}^{h} T_0(h-s)D_0 S(s)By\dd s-hT_0(h-s_0)D_0 S(s_1)By\Bigr\|\le C h^{1+\gamma}(\|By\|+\|B^2 y\|).
\end{equation*} 
\end{lemma}
\begin{proof}
For any $y\in\dom(B^2)$ we can write
\begin{align*}
 &\Bigl\|\int_{0}^{h} T_0(h-s)D_0 S(s)By\dd s-hT_0(h-s_0)D_0 S(s_1)By\Bigr\|\\
 &\qquad=\Bigl\|\int_{0}^{h} \bigl( T_0(h-s)D_0 S(s)By-T_0(h-s_0)D_0 S(s_1)By\bigr)\dd s\Bigr\|\\
 &\qquad\le \Bigl\|\int_{0}^{h} \bigl(T_0(h-s)-T_0(h-s_0)\bigr)D_0 S(s)By\dd s\Bigr\|\\
 &\qquad\quad+ \Bigl\|\int_{0}^{h} T_0(h-s_0)D_0 \bigl(S(s)- S(s_1)\bigr)By\dd s\Bigr\|.
\end{align*}
The second term can be further estimated as follows:
\begin{equation}\label{eq:B2}
 \begin{split}
 & \Bigl\|\int_{0}^{h} T_0(h-s_0)D_0 \bigl(S(s)- S(s_1)\bigr)By\dd s\Bigr\|\\
 & \qquad\le C_1\int_{0}^{h} \|\bigl(S(s)- S(s_1)\bigr)By\|\dd s
 = C_1\int_{0}^{h}\Bigl\|\int_{s_1}^s S(r)B^2y\dd r\Bigr\|\dd s \\
 & \qquad\le C_2h^2\|B^2 y\|.
 \end{split}
\end{equation}
It remains to estimate the first term. Since $(-A_0)^\gamma D_0$ is closed and everywhere defined, it is bounded (see Remark \ref{rem:fracpowbdd2}) and hence we can write
\begin{equation*}
 \begin{split}
 &\Bigl\|\int_0^h \bigl(T_0(h-s)-T_0(h-s_0)\bigr)D_0 S(s)By\dd s\Bigr\|\\
 &\qquad\le \int_0^h \|(-A_0)^{-\gamma} \bigl(T_0(h-s)-T_0(h-s_0)\bigr)(-A_0)^{\gamma}D_0 S(s)By\|\dd s\\
 &\qquad\le \int_0^h \|(-A_0)^{-\gamma} \bigl(T_0(h-s)-T_0(h-s_0)\bigr)\|\cdot \|(-A_0)^{\gamma}D_0 S(s)By\|\dd s\\
 &\qquad\le C_3\|By\|\int_0^h \|(-A_0)^{-\gamma}\bigl(T_0(h-s)-T_0(h-s_0)\bigr)\|\dd s.
 \end{split}
\end{equation*}
Now, by Remark~\ref{rem:fracpowbdd} we have
\begin{equation*}
 \|(-A_0)^{-\gamma }\bigl(T_0(h-s)-T_0(h-s_0)\bigr)\|
  \le C_4 |s-s_0|^\gamma.
\end{equation*}
 Inserting this back into the previous inequality and integrating with respect to $s$ we finally obtain the statement. 
\end{proof}

The next result yields that the order of Lie splitting is (at most $1$ but) as near to $1$ as we wish, provided the initial data is smooth enough.

\begin{theorem}[Convergence of the Lie splitting]\label{thm:globlieA}
	Let $A_0$ and $B$ be the generator of the strongly continuous semigroups $(T_0(t))_{t\ge0}$ and $(S(t))_{t\ge0}$, respectively. Suppose Hypotheses \ref{hyp:boundary}, \ref{hyp:inv}, \eqref{eq:Qtb}, \ref{hyp:sec}. 
For each $\tmax>0$ there is $C\ge 0$ such that for every $n\in\NN$, $y\in\dom(B^2)$ and $t\in [0,\tmax]$ we have
\begin{equation*}
 \Bigl\|V_n^\text{\Lie}(\tfrac tn)y +\int_{0}^{t} T_0(t-s)D_0 S(s)By\dd s\Bigr\|\le C\frac{t\log(n)}{n}(\|By\|+\|B^2 y\|).
\end{equation*}
\end{theorem}
\begin{proof}
With $\tau=\frac tn$ we have
\begin{equation*}
 \begin{split}
 V_n^\text{\Lie}(\tau)y &+\int_{0}^{t} T_0(t-s)D_0 S(s)By\dd s\\
 &=-\tau\sum\limits_{j=0}^{n-1} T_0((n-1-j)\tau)T_0(\tau)D_0BS(\tau)S(j\tau)y \\  
 &\qquad+\int_{0}^{t} T_0(t-s)D_0 S(s)By\dd s\\
 &=-\sum\limits_{j=0}^{n-1} \Bigl(\tau T_0((n-1-j)\tau)T_0(\tau)D_0BS(\tau)S(j\tau)y\\ &\qquad -\int_{j\tau}^{(j+1)\tau} T_0(t-s)D_0 S(s)By\dd s\Bigr).
 \end{split}
\end{equation*}
Notice that for $j\in\{0,\ldots,n-1\}$ we have
\begin{align*}
 &\int_{j\tau}^{(j+1)\tau} T_0(t-s)D_0 S(s)By\dd s-\tau T_0((n-1-j)\tau)T_0(\tau)D_0BS(\tau)S(j\tau)y \\
 &=T_0(t-(j+1)\tau)\int_{j\tau}^{(j+1)\tau} T_0((j+1)\tau-s)D_0 S(s)By\dd s\\
 &\quad-\tau T_0(t- (j+1)\tau)T_0(\tau)D_0S((j+1)\tau)By.
\end{align*}
If $j\in\{0,\ldots,n-2\}$, then by Lemma~\ref{lem:loclieA} we conclude that
\begin{align*}
 &\Bigl\|\int_{j\tau}^{(j+1)\tau} T_0(t-s)D_0 S(s)By\dd s-\tau T_0((n-1-j)\tau)T_0(\tau)D_0BS(\tau)S(j\tau)y \Bigr\|\\
 &\le \|A_0T_0(t- (j+1)\tau)\|\cdot \Bigl\|\int_{j\tau}^{(j+1)\tau} T_0((j+1)\tau-s)A_0^{-1}D_0 S(s)By\dd s\\
 &\quad\quad\quad-\tau T_0(\tau)A_0^{-1}D_0S((j+1)\tau)By\Bigr\|\\
 &\le C_1\frac{1}{t-(j+1)\tau}\Bigl\|\int_{0}^{\tau} T_0(\tau-s)A_0^{-1}D_0 S(s+j\tau)By\dd s\\
 &\quad\quad\quad-\tau T_0(\tau)A_0^{-1}D_0S(\tau)S(j\tau)By\Bigr\|\\ 
 &\le C_2 \frac{1}{t-(j+1)\tau}\tau^{2}(\|BS(j\tau)y\|+\|B^2 S(j\tau)y\|)\\
 &\le C_3\frac{t}{n(n-(j+1))}(\|By\|+\|B^2 y\|).
\end{align*}
Whereas for $j=n-1$ we have by Lemma~\ref{lem:D0fracpow} (with $\gamma=0$, $h=\tau$, $s_0=s_1=\tau$) that
\begin{align*}
&\Bigl\|\int_{j\tau}^{(j+1)\tau} T_0(t-s)D_0 S(s)By\dd s-\tau T_0((n-1-j)\tau)T_0(\tau)D_0BS(\tau)S(j\tau)y \Bigr\|\\
&=\Bigl\|\int_{0}^{\tau} T_0(\tau-s)D_0 S((n-1)\tau+s)By\dd s-\tau T_0(\tau)D_0S(\tau)S((n-1)\tau)By\Bigr\|\\
&\le C_4 \frac tn (\|By\|+\|B^2 y\|).
\end{align*}
Summing these terms for $j=0,\ldots, n-1$ we obtain that
\begin{align*}
&\Bigl\|V_n^\text{\Lie}(\tau)y +\int_{0}^{t} T_0(t-s)D_0 S(s)By\dd s\Bigr\|\\
&\le \sum_{j=0}^{n-2} C_3\frac{t}{n(n-(j+1))}(\|By\|+\|B^2 y\|)+C_4 \frac tn (\|By\|+\|B^2 y\|)\\
&\le C\frac{t\log(n)}{n}(\|By\|+\|B^2 y\|),
\end{align*}
as asserted.
\end{proof}

\begin{lemma}[Local error of the Strang splitting]\label{lem:strlem}
	Let $A_0$ and $B$ be the generator of the strongly continuous semigroups $(T_0(t))_{t\ge0}$ and $(S(t))_{t\ge0}$, respectively. Suppose Hypotheses \ref{hyp:boundary}, \ref{hyp:inv}, \eqref{eq:Qtb}, \ref{hyp:sec} and also
 \ref{hyp:fracpow}, i.e., that $\rg(D_0)\subseteq\dom((-A_0)^\gamma)$ for some $\gamma\in[0,1]$.
For every $\tmax>0$ there is $C\ge0$ such that for every $h\in [0,\tmax]$ and for every $y\in \dom(B^3)$ we have
\begin{multline*}
\Bigl\|\int_{0}^{h} T_0(h-s)A_0^{-1}D_0 S(s)By\dd s-hT_0(\hpk)A_0^{-1}D_0 S(\hpk)By\Bigr\|
\le C h^{2+\gamma}(\|By\|+\|B^2 y\|+\|B^3 y\|).
\end{multline*}
\end{lemma}
\begin{proof}
We have
\begin{align*}
&\int_{0}^{h} T_0(h-s)A_0^{-1}D_0 S(s)By\dd s-hT_0(\hpk)A_0^{-1}D_0 S(\hpk)By\\
&=\int_{0}^{h} T_0(h-s)A_0^{-1}D_0 S(s)By-T_0(\hpk)A_0^{-1}D_0 S(\hpk)By\dd s\\
&=\int_{0}^{h} T_0(h-s)A_0^{-1}D_0 \bigl(S(s)-S(\hpk)\bigr)By\dd s\\
&\quad+\int_{0}^{h} \bigl(T_0(h-s)-T_0(\hpk)\bigr)A_0^{-1}D_0 S(\hpk)By \dd s\\
&=\int_{0}^{h} \Bigl(T_0(h-s)-T_0(\hpk)\Bigr) A_0^{-1}D_0 \bigl(S(s)-S(\hpk)\bigr)By\dd s\\
&\quad+\int_{0}^{h} T_0(\hpk)A_0^{-1}D_0 \bigl(S(s)-S(\hpk)\bigr)By\dd s\\
&\quad+\int_{0}^{h} \bigl(T_0(h-s)-T_0(\hpk)\bigr)A_0^{-1}D_0 S(\hpk)By \dd s=I_1+I_2+I_3,
\end{align*}
where $I_1,I_2,I_3$ denote the integrals on the right-hand side in the respective order of appearance.

We start with the estimation of $I_1$. 
Inserting the Taylor remainder 
\begin{equation*}
 (S(s)-S(\hpk))By=\int_{\hpk}^s S(r)B^2 y\dd r,
\end{equation*} 
 and the analogous formula for $T_0(h-s)-T_0(\hpk)$, 
 in the definition of $I_1$ yields that
\begin{align*}
 I_1&=\int_{0}^{h} \Bigl(T_0(h-s)-T_0(\hpk)\Bigr) A_0^{-1}D_0 \bigl(S(s)-S(\hpk)\bigr)By\dd s\\
 &=\int_0^h\int_{\hpk}^{h-s} T_0(t)A_0A_0^{-1}D_0\int_{\hpk}^s S(r)B^2 y\dd r \dd t \dd s\\
 &=\int_0^h\int_{\hpk}^{h-s} T_0(t)D_0\int_{\hpk}^s S(r)B^2 y \dd r \dd t \dd s.
\end{align*}
Whence we conclude
\begin{align}
\label{eq:I1} \|I_1\|&\le \int_0^h\Bigl|\int_{\hpk}^{h-s} \|T_0(t)D_0\|\Bigl|\int_{\hpk}^s \|S(r)\|\|B^2 y\|\dd r \Bigr|\dd t\Bigr| \dd s\\
\notag&\le C_1\|D_0\|\int_0^h |h-s-\hpk|\cdot |\hpk-s| \dd s \cdot \|B^2 y\| \le C_2 h^{3} \|B^2 y\|,
\end{align}
where $C_1$ and $C_2$ depend only on the growth bounds of $(T_0(t))_{t\ge}$ and $(S(t))_{t\ge 0}$ and on $\tmax$ and $\|D_0\|$.

The next is the estimation of the integral $I_2$. 
Now instead of inserting a first order Taylor approximation for $S(s)$ in the definition of $I_2$ we make use of the special structure of the Strang splitting and recall the following Taylor formula
\begin{align*}
 S(s)By=S(\hpk)By+(s-\hpk)S(\hpk)B^2 y+\int_0^{s-\hpk}(s-\hpk-r) S(\hpk+r)B^3
y\dd r.
\end{align*}
If we substitute this into the definition of $I_2$, we arrive at
\begin{align*}
 I_2&=\int_{0}^{h} T_0(\hpk)A_0^{-1}D_0 \bigl(S(s)-S(\hpk)\bigr)By\dd s\\
 &=T_0(\hpk)A_0^{-1}D_0 \int_{0}^{h} \bigl(S(s)-S(\hpk)\bigr)By\dd s\\
 &=T_0(\hpk)A_0^{-1}D_0 \int_{0}^{h} \Bigl(S(\hpk)By+(s-\hpk)S(\hpk)B^2 y\\
 &\quad +\int_0^{s-\hpk}(s-\hpk-r) S(r)S(\hpk)B^3
 y\dd r-S(\hpk)By\Bigr)\dd s\\
 &=T_0(\hpk)A_0^{-1}D_0 \Bigl(\int_{0}^{h} (s-\hpk)S(\hpk)B^2 y\dd s\\
 &\quad +\int_0^h \int_0^{s-\hpk}(s-\hpk-r) S(r)S(\hpk)B^3
 y\dd r\dd s\Bigr)\\
 &=T_0(\hpk)A_0^{-1}D_0\int_0^h \int_0^{s-\hpk}(s-\hpk-r) S(r)S(\hpk)B^3
 y\dd r\dd s,
\end{align*}
the last equality being true since the first integral on the right-hand side on the line before is $0$.
This immediately implies the desired norm-estimate for $I_2$:
\begin{align*}
 \|I_2\|&\le \frac12\|T_0(\hpk)A_0^{-1}D_0\|\cdot \Bigl\|\int_0^h 
 \int_0^{s-\hpk}(s-\hpk-r) S(r)S(\hpk)B^3 y\dd r\dd s\Bigr\|\\
 &\le C_3 h^3\|B^3 y\|,
\end{align*}
where $C_3$ is an appropriate constant independent of $y$ and $h\in [0,\tmax]$.

We finally turn to the estimation of the term $I_3$, and this is only where the order reduction by $1-\gamma$ occurs.
If we abbreviate $z=D_0 S(\hpk)By$, then
\begin{equation*}
  I_3=\int_0^h \bigl(T_0(h-s)-T_0(\hpk)\bigr)A_0^{-1}z.
\end{equation*}
By analyticity we have $T_0(\hpk) z\in \dom(A_0^2)$ so, similarly to the case of the term $I_2$, we can use the Taylor expansion for $0\leq s<h$
\begin{align*}
  T_0(h-s)A_0^{-1}z&=T_0(\hpk)A_0^{-1}z+(\hpk-s)A_0T_0(\hpk)A_0^{-1}z\\
 &\quad +A_0\int_0^{\hpk-s}(\hpk-s-r) A_0T_0(\hpk+r)A_0^{-1}z\dd r. 
\end{align*}
Whence we conclude
\begin{align*}
 I_3&=\int_{0}^{h} \bigl(T_0(h-s)-T_0(\hpk)\bigr)A_0^{-1}z\dd s\\
& =\int_{0}^{h} \Bigl((\hpk-s)T_0(\hpk)z\\
 &\quad +A_0\int_0^{\hpk-s}(\hpk-s-r) T_0(\hpk)T_0(r)z\dd r\Bigr)\dd s,
 \intertext{since the first integral here is $0$, we arrive at}
 I_3&=A_0\int_{0}^{h}\int_0^{\hpk-s}(\hpk-s-r) T_0(\hpk)T_0(r)D_0 S(\hpk)By\dd r\dd s.
\end{align*}
We take the norm here and, using Remark \ref{rem:fracpowbdd}, estimate trivially:
\begin{align}\label{eq:I3}
 \|I_3\|&=\Bigl\|\int_{0}^{h}(-A_0)^{1-\gamma}\int_0^{\hpk-s}(\hpk-s-r) T_0(\hpk)T_0(r)(-A_0)^{\gamma}D_0 S(\hpk)By\dd r\dd s\Bigr\|\\
 \notag &\le C_4 \|(-A_0)^{1-\gamma}T_0(\hpk)\|\cdot \|(-A_0)^{\gamma}D_0 \|\cdot \|By\| \int_{0}^{h}\Bigl|\int_0^{\hpk-s} |\hpk-s-r|\dd r\Bigr| \dd s\\
 \notag&=C_5 \frac{h^3}{h^{1-\gamma}}\|By\|=C_5 h^{2+\gamma}\|By\|.
\end{align}
The proof of the lemma is now finished by putting together the estimates for $I_1$, $I_2$ and $I_3$.
\end{proof}

\begin{theorem}[Convergence of the Strang splitting]\label{thm:strconv}
	Let $A_0$ and $B$ be the generator of the strongly continuous semigroups $(T_0(t))_{t\ge0}$ and $(S(t))_{t\ge0}$, respectively. Suppose Hypotheses \ref{hyp:boundary}, \ref{hyp:inv}, \eqref{eq:Qtb}, \ref{hyp:sec} and also
 \ref{hyp:fracpow}, i.e., that $\rg(D_0)\subseteq\dom((-A_0)^\gamma)$ for some $\gamma\in[0,1]$.
For each $\tmax>0$ there is $C\ge 0$ such that for every $n\in \NN$, $y\in \dom(B^3)$ and $t\in [0,\tmax]$ we have
\begin{equation*}
\Bigl\|V_n^\text{\Str}(\tfrac tn)y +\int_{0}^{t} T_0(t-s)D_0 S(s)By\dd s\Bigr\|\le C\frac{t^{1+\gamma}\log(n)}{n^{1+\gamma}}(\|By\|+\|B^2 y\|+\|B^3 y\|).
\end{equation*}
\end{theorem}

\begin{proof}
Set $\tpn:=\frac tn$.
Recall from \eqref{eq:StrV} that for $y\in \dom(B)$ we have
\begin{align*}
	V_n^\text{\Str}(\tpn)y = - \tpn\sum\limits_{j=0}^{n-1}
	T_0((n-j-\tfrac 12)\tpn)D_0S((j+\tfrac 12)\tpn)By,
\end{align*}
so that
\begin{align}
 \notag \int_{0}^{t} &T_0(t-s)D_0 S(s)By\dd s+V_n^\text{\Str}(\tpn)y \\
 \notag &\quad = \sum\limits_{j=0}^{n-1}\int_{j\tpn}^{(j+1)\tpn }T_0(t-s)D_0 S(s)By-T_0((n-j-\tfrac 12)\tpn)D_0S((j+\tfrac 12)\tpn)By\dd s\\
\label{eq:StrV2} &\quad=\sum\limits_{j=0}^{n-1}T_0((n-j-1)\tpn)\int_{0}^{\tpn }T_0(\tpn-s)D_0 S(s)S(j\tpn)By\\
\notag&\qquad\qquad\qquad\qquad\qquad\qquad\qquad\qquad-T_0(\tpnh)D_0S(\tpnh)S(j\tpn)By\dd s.
\end{align}
We first consider the term for $j=n-1$. 
By Lemma~\ref{lem:D0fracpow}, with $h=\tau$ and $s_0=s_1=\tpnh$ we have that 
\begin{align*}
 &\Bigl\|\int_{0}^{\tpn }T_0(\tpn-s)D_0 S(s)S((n-1)\tpn)By-T_0(\tpnh)D_0S(\tpnh)S((n-1)\tpn)By\dd s\Bigr\|\\ 
 &\qquad\le C_1 \tpn^{1+\gamma}(\|BS((n-1)\tpn)y\|+\|B^2S((n-1)\tpn)y\|)\\
 &\qquad\le C_2\tpn^{1+\gamma}(\|By\|+\|B^2y\|)
\end{align*}
for $t\in [0,\tmax]$.
Next we consider the summands in \eqref{eq:StrV2} for $j\in \{0,\ldots,n-2\}$. 
In these cases we can write
\begin{equation*}
\begin{split}
 &\Bigl\|T_0\bigl((n-j-1)\tpn\bigr)\int_{0}^{\tpn }T_0(\tpn-s)D_0 S(s)S(j\tpn)By-T_0(\tpnh)D_0S(\tpnh)S(j\tpn)By\dd s\Bigr\|\\
 &\quad=\|A_0T_0\bigl((n-j-1)\tpn\bigr)\int_{0}^{\tpn }T_0(\tpn-s)A_0^{-1}D_0 S(s)S(j\tpn)By\\
 &\quad\quad\quad -T_0(\tpnh)A_0^{-1}D_0S(\tpnh)S(j\tpn)By\dd s\Bigr\|\\
 &\quad\le\|A_0T_0\bigl((n-j-1)\tpn\bigr)\|\cdot \Bigl\|\int_{0}^{\tpn }T_0(\tpn-s)A_0^{-1}D_0 S(s)S(j\tpn)By\\
 &\quad\quad\quad -T_0(\tpnh)A_0^{-1}D_0S(\tpnh)S(j\tpn)By\dd s\Bigr\|, 
 \end{split}
\end{equation*} 
and by Lemma~\ref{lem:strlem} and Remark~\ref{rem:fracpowbdd} we can continue as follows:
 \begin{equation*}
\begin{split}
 &\quad\le \frac{C_3}{(n-j-1)\tpn}\tpn^{2+\gamma}\bigl(\|BS(j\tpn)y\|+\|B^2S(j\tpn)y\|+\|B^3 S(j\tpn)y\|\bigr)\\
 &\quad\le \frac{C_4}{(n-j-1)\tpn}\tpn^{2+\gamma}\bigl(\|By\|+\|B^2y\|+\|B^3 y\|\bigr)
\end{split}
\end{equation*}
for constants $C_3,C_4$ independent of $y$, $n$ and $t\in[0,\tmax]$. 
Summing up these estimates we arrive at
\begin{equation*}
\begin{split}
 &\Bigl\|\int_{0}^{t} T_0(t-s)D_0 S(s)By\dd s+V_n^\text{\Str}(\tpn)y \Bigr\|\\
 &\quad\le C_2 \tpn^{1+\gamma}\bigl(\|By\|+\|B^2y\|\bigr)+\sum\limits_{j=0}^{n-2} \frac{C_4}{(n-j-1)\tpn}\tpn^{2+\gamma}\bigl(\|By\|+\|B^2y\|+\|B^3 y\|\bigr)\\
 &\quad\le \Bigl(C_2\frac{t^{1+\gamma}}{n^{1+\gamma}}+\frac{C_4t^{2+\gamma}}{n^{2+\gamma}}\sum\limits_{k=1}^{n-1}\frac{n}{tk}\Bigr)\cdot\bigl (\|By\|+\|B^2y\|+\|B^3 y\|\bigr)\\
 &\quad \le C\frac{t^{1+\gamma}\log(n)}{n^{1+\gamma}}\bigl(\|By\|+\|B^2y\|+\|B^3 y\|\bigr),
\end{split}
\end{equation*}
with an appropriate constant $C\ge 0$.
The proof is complete.
\end{proof}


Finally, let us turn to the weighted splittings. For any $\Theta\in [0,1]$ the weighted splitting possess at least the convergence properties as the Lie splitting. For the case $\Theta=1/2$ one can prove even more.
\begin{lemma}[Local error of the symmetrically weighted Splitting]\label{lem:weilem}
	Let $A_0$ and $B$ be the generator of the strongly continuous semigroups $(T_0(t))_{t\ge0}$ and $(S(t))_{t\ge0}$, respectively. Suppose Hypotheses \ref{hyp:boundary}, \ref{hyp:inv}, \eqref{eq:Qtb}, \ref{hyp:sec} and also
 \ref{hyp:fracpow}, i.e., that $\rg(D_0)\subseteq\dom((-A_0)^\gamma)$ for some $\gamma\in[0,1]$.
For every $\tmax>0$ there is $C\ge 0$ such that for every $h\in [0,\tmax]$ and for every $y\in \dom(B^3)$ we have
\begin{multline*}
 \Bigl\|\int_{0}^{h} T_0(h-s)A_0^{-1}D_0 S(s)By\dd s-\frac12\Bigl(h T_0(h)A_0^{-1}D_0BS(h)y+hA_0^{-1}D_0By\Bigr)\Bigr\|\\
 \quad\le C h^{2+\gamma}(\|By\|+\|B^2 y\|+\|B^3 y\|).
\end{multline*}
\end{lemma}
\begin{proof}
We have that
\begin{align*}
&2\int_{0}^{h} T_0(h-s)A_0^{-1}D_0 S(s)By\dd s-\Bigl(h T_0(h)A_0^{-1}D_0BS(h)y+hA_0^{-1}D_0By\Bigr)=\\
&\quad =\int_{0}^{h} T_0(h-s)A_0^{-1}D_0 S(s)By- T_0(h)A_0^{-1}D_0BS(h)y\dd s\\ 
&\qquad +\int_{0}^{h} T_0(h-s)A_0^{-1}D_0 S(s)By -A_0^{-1}D_0By\dd s\\
&\quad=\int_{0}^{h}T_0(h)A_0^{-1}D_0(2S(s)- S(h)-\Id)By\dd s\\
&\qquad+\int_{0}^{h} (2T_0(h-s)-T_0(h)-\Id)A_0^{-1}D_0 S(h)By\dd s\\
&\qquad+\int_{0}^{h} (2T_0(h-s)-T_0(h)-\Id)A_0^{-1}D_0 (S(s)-S(h))By\dd s\\
&\qquad+\int_0^h (\Id-T_0(h))A_0^{-1}D_0(S(s)-\Id)By\dd s=I_1+I_2+I_3+I_4,
\end{align*}
where $I_j$ denotes the respective term on the right-hand side in order of occurrence. 
We first consider the term $I_1$. Since $y\in \dom(B^3)$ for any $t>0$ we have the Taylor expansion
\begin{align*}
 S(t)By=By+tB^2y+\int_0^t (t-r)S(r)B^3y\dd r.
\end{align*}
Substituting this into the  formula for $I_1$ we obtain that
\begin{align*}
 I_1&=T_0(h)A_0^{-1}D_0\int_0^h \Bigl(2By+2sB^2y+2\int_0^s (s-r)S(r)B^3y\dd r- By-hB^2y\\
 &\qquad -\int_0^h (h-r)S(r)B^3 y\dd r-By\Bigr)\dd s\\
 &=T_0(h)A_0^{-1}D_0\int_0^h \Bigl[(2s-h)B^2y+ 2\int_0^s(s-r)S(r)B^3y\dd r\\
 &\qquad- \int_0^h(h-r)S(r)B^3 y\dd r\Bigr]\dd s\\
 &=T_0(h)A_0^{-1}D_0\int_0^h \Bigl[2\int_0^s(s-r)S(r)B^3y\dd r-\int_0^h(h-r)S(r)B^3 y\dd r\Bigr]\dd s,
\end{align*}
since the integral of the first term is $0$.
Whence we conclude
\begin{align*}
 \|I_1\|\le C_3 h^3 \|B^3y\|.
\end{align*}
The term $I_2$ can be estimated similarly, and we obtain
\begin{align*}
 \|I_2\|&\le C_1 h^{2+\gamma}(\|By\|+\|B^2y\|),
\end{align*}
cf.\ \eqref{eq:I3} in the estimation of the term $I_3$ in Lemma~\ref{lem:strlem}.
Finally, for the terms $I_3$  and $I_4$ we have
\begin{align*}
 \|I_3\|,\:\|I_4\|&\le C_2 h^{3}\|B^2y\|,
\end{align*}
cf.\ \eqref{eq:I1} in the estimation of the term $I_1$ in Lemma~\ref{lem:strlem}. 
Putting the estimates for the terms $I_1$, $I_2$, $I_3$, $I_4$ together 
finishes the proof of the lemma.
\end{proof}

Based on Lemma~\ref{lem:weilem} we immediately obtain the following error estimate for the symmetrically weighted splitting, the proof is analogous to the one of Theorem~\ref{thm:strconv}.

\begin{theorem}[Convergence of the symmetrically weighted splitting]\label{thm:symwghconv}
	Let $A_0$ and $B$ be the generator of the strongly continuous semigroups $(T_0(t))_{t\ge0}$ and $(S(t))_{t\ge0}$, respectively. Suppose Hypotheses \ref{hyp:boundary}, \ref{hyp:inv}, \eqref{eq:Qtb}, \ref{hyp:sec} and also
 \ref{hyp:fracpow}, i.e., that $\rg(D_0)\subseteq\dom((-A_0)^\gamma)$ for some $\gamma\in[0,1]$.
For each $\tmax>0$ there is $C\ge 0$ such that for every $n\in \NN$, $y\in \dom(B^3)$ and $t\in [0,\tmax]$ we have
\begin{equation*}
 \Bigl\|V_n^\text{\wgh}(\tfrac tn)y +\int_{0}^{t} T_0(t-s)D_0 S(s)By\dd s\Bigr\|
 \le C\frac{t^{1+\gamma}\log(n)}{n^{1+\gamma}}(\|By\|+\|B^2 y\|+\|B^3 y\|).
\end{equation*}
\end{theorem}

\section{Numerical examples}
\label{sec:num}

We present two examples to illustrate the theoretical results concerning the accuracy of the splitting procedures applied to problem \eqref{eq:main0}. In both cases, we analyse the order of the global error by solving the problems using various values of the (splitting) time step $\tau>0$. Fitting a straight line to these data in the logarithmic scale, the resulting slope yields the computational (numerical) order of the method. This is then to be compared to the theoretical order obtained in Theorems \ref{thm:globlieA}, \ref{thm:strconv}, and \ref{thm:symwghconv} for the Lie \eqref{eq:lie}, Strang \eqref{eq:str}, and weighted splittings \eqref{eq:wgh}, respectively.

For such illustration purposes the simplest non-trivial example serves as a good basis. In fact, we consider the heat equation on the domain $(0,\beta)$, $\beta>0$ and a system of ordinary differential equations on the boundary for the unknown function $w\colon[0,\infty)\times[0,\beta]\to\RR$:
\begin{equation}\label{eq:example0pde}
\left\{
\begin{aligned}
\partial_tw(t,x) &= c\partial_{xx}w(t,x) && \text{ for } t>0,\:x\in(0,\beta), \\
w(0,x) &= w_0(x)&& \text{ for } x\in (0,\beta), \\
\dot w(t,0) &=b_{11}w(t,0) + b_{12}w(t,\beta)&& \text{ for } t\ge 0, \\
\dot w(t,\beta) &= b_{21}w(t,0) + b_{22}w(t,\beta)&& \text{ for } t\ge 0.
\end{aligned}
\right.
\end{equation}
As explained in Section \ref{sec:abstractBCs} this equation can be casted in the form of \eqref{eq:main0}. The occurring spaces and operators are as follows (cf.{} Example \ref{examp:beltrami}):
\begin{itemize}
	\item $E=\Ell^2(0,\beta)$, $F=\CC^2$.
	\item $A_m=c\Delta$, and $\Delta$ is the distributional Laplace operator on $\dom(A_m)=\{f\in \He^{1/2}(0,\beta)\st f''\in \Ell^2(0,\beta)\}$.
	\item $L$ is the Dirichlet trace on $\{0,\beta\}$. 
	\item $A_0$ is the scalar multiple of the Dirichlet Laplace operator generating the Dirichlet heat semigroup $(T_0(t))_{t\geq 0}$.
	\item $D_0:\CC^2\to H^{1/2}(0,\beta)$, $D_0(a,b)(r)=br/\beta+a(\beta-r)/\beta$.
	\item The operator $B$ is given by \[B=\begin{pmatrix}b_{11}&b_{12}\\b_{21}&b_{22}\end{pmatrix}\in \CC^{2\times 2}\]
	generating the semigroup $(S(t))_{t\geq 0}=(\ee^{tB})_{t\geq 0}$.
\end{itemize}

\subsection*{Implementation}
Many ingredients, needed for the splitting methods, can be calculated explicitly, such as $D_0$, $(S(t))_{t\geq 0}$. For the implementation we used the corresponding built-in \textsc{matlab} functions. For the solution of the heat equation with \textit{homogeneous} Dirichlet boundary condition, i.e., for determining $y(t)=T_0(t)y(0)$, it is plausible to apply an appropriate spectral method, in this case the Fourier method and use expansion with respect to the orthogonal basis of eigenfunctions of $A_0$. This is implemented by using \textsc{matlab}'s built in \texttt{fft} function (Fast Fourier Transform). To this end we choose an integer $N_x$ and split the space interval $[0,\beta]$ into $N_x$ pieces of equal length, and the Fourier (sine) series is also truncated at $N_x$. We compute the values $u(t_n)$ only at grid points $x_j=j\beta/N_x$ for all $j=1,\ldots,N_x$ and for all time levels, $t_n=n\tau\in[0,\tmax]$, $n=1,\ldots,N_t:=\tmax/\tau$ for some $\tmax>0$ (such that $N_t$ is an integer). Similarly, the boundary values $v(t_n)$ are also computed only at the same time levels. For all $n=1,\ldots,N_t$, the numerical solution $\vecu_n^{N_x}\in\RR^{N_x}$ has then the elements $(\vecu_n^{N_x})_j$ being the approximation to $(\vecu(t_n))(x_j)$ for all $j=1,\ldots,N_x$.

The relative global error of the splittings is then computed as
\begin{equation}\label{eq:error}
\varepsilon(\tau):=\frac{\|\vecu(\tmax)-\vecu_{N_t}^{N_x}\|_2}{\|\vecu(\tmax)\|_2}
\end{equation}
with the Euclidean vector norm.
The results of Section \ref{sec:order} forecast the following behaviour (for fixed $\tmax>0$ and initial value)
\begin{align*}
	\text{for the sequential splitting}&\quad\varepsilon(\tau)=O(\tau|\log(\tau)|)\\
	\left.\begin{array}{ll}\text{for the symmetrically weighted}\\
	\text{and Strang splittings}\end{array}\right\}&\quad\varepsilon(\tau)=O(\tau^{1+\gamma}),
\end{align*}
where $\gamma\in [0,1/4)$ is as near to $1/4$ as we want, see Remark \ref{rem:fracpowbdd2}. For the computational orders of these methods we therefore expect 1.0 and 1.25, respectively. In the next two examples we test the methods against these expectations.

\subsection*{Exponential growth and decay on the boundary}
We consider the following specific problem:
\begin{equation}\label{eq:example1pde}
\left\{
\begin{aligned}
\partial_tw(t,x) &= \partial_{xx}w(t,x) && \text{ for } t>0,\:x\in(0,\pi), \\
w(0,x) &= \cos(\tfrac x2)+\sinh(x) && \text{ for } x\in[0,\pi], \\
\dot w(t,0) &= -\tfrac 14w(t,0) && \text{ for } t\ge 0, \\
\dot w(t,\pi) &= w(t,\pi) && \text{ for } t\ge 0.
\end{aligned}
\right.
\end{equation}
This corresponds to the choices $\beta=\pi$, $c=1$ and
\[
B=\begin{pmatrix}-\tfrac 14&0\\0&1\end{pmatrix}.
\]
For this test example, the exact solution \[   
w(t,x) = \ee^{-t/4}\cos(\tfrac x2)+\ee^{t}\sinh(x)\]
 is known for all $t\ge 0$ and $x\in[0,\pi]$, therefore, we have
\begin{align}
\label{eq:exact1u} (u(t))(x) &= \ee^{-t/4}\cos(\tfrac x2)+\ee^{t}\sinh(x) \quad\text{for}\quad x\in(0,\pi), \\
\label{eq:exact1v} v(t) &= \binom{\ee^{-t/4}}{\ee^{t}\sinh(\pi)}
\end{align}
for all $t\ge 0$. Based on this, Figure \ref{fig:exm1sol} shows the exact solution $u(t)$ at certain time levels, while in Figure \ref{fig:exm1bdry} the time evolution of the boundary values $v(t)$ is presented.

\begin{figure}[!ht]
\centering
\rotatebox{-90}{\includegraphics[width=6cm]{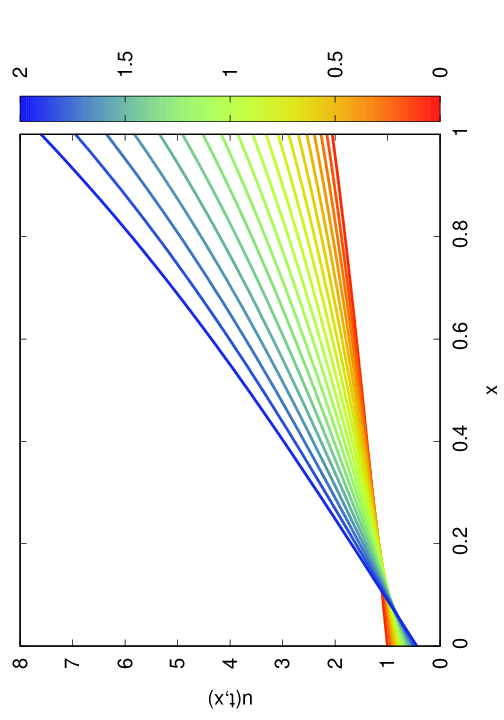}}
\caption{Exact solution $u(t)$ of Example 1 for time levels $t_n=n\cdot 0.1$, $n=1,\ldots,20$. Red lines correspond to the beginning and blue ones to the end. \label{fig:exm1sol}}
\end{figure}

\begin{figure}[!ht]
\centering
\rotatebox{-90}{\includegraphics[width=6cm]{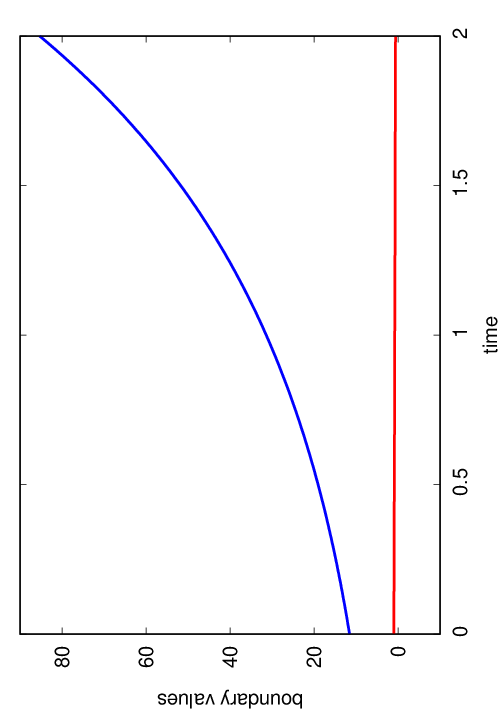}}
\caption{Exact solution $v(t)$ of Example 1 for all $t\in[0,2]$. The red line corresponds to the first coordinate function (i.e., the left boundary value), while the blue line to the second one (right boundary value). \label{fig:exm1bdry}}
\end{figure}

In Figure \ref{fig:exm1order}, the order plots of the splittings are shown: the error values $\log(\varepsilon(\tau))$ via several values of $\log(\tau)$. The slope of the lines corresponds to the approximation of the order (called computational order later on). For comparison purposes, we also included the lines with slope $1$ and $1.25$.

\begin{figure}[!ht]
\centering
\rotatebox{-90}{\includegraphics[width=6cm]{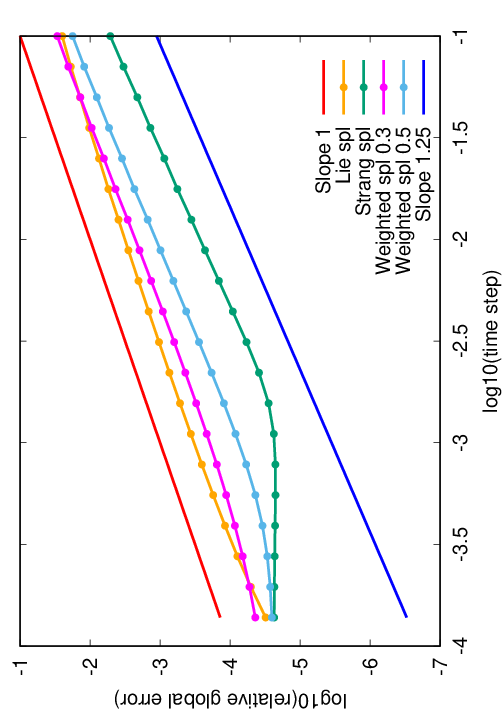}}
\caption{Order plots of the various splitting procedures for problem \eqref{eq:example1pde} with $N_x=32768$. \label{fig:exm1order}}
\end{figure}

One can see that for each splitting, there exists a value of $\log(\tau)$ below which the error no longer decreases. At that point the global error is  dominated by the error of the spatial discretisation (the implementation/truncation error of the Fourier method in this case) and not the splitting error anymore. When computing the order, we fit a straight line only to the relevant data. The resulting values of the slopes are listed in Table \ref{tab:exm1order}.

\begin{table}[!ht]
\centering
\caption{Computational orders of the splitting procedures computed as the slope of the fitted line for $N_x=32768$ in the case of problem \eqref{eq:example1pde}. 
 \label{tab:exm1order}}
{\footnotesize
\begin{tabular}{l|c|c|c|c}
splitting & Lie & Strang & Weighted $\Theta=0.3$ & Weighted $\Theta=0.5$ \\\hline
analytical order & $\sim1$ & $\sim1.25$ & $\sim1$ & $\sim1.25$ \\
computational order & $1.0004$ & $1.2405$ & $1.0549$ & $1.1646$
\end{tabular}}
\end{table}
Note that the asymmetrically weighted splitting is not studied in this paper analytically, but the numerical findings coincide with what one might expect: the same behaviour as for the Lie splitting. For each of the splitting methods, as one expects, the larger the $N_x$ values are, the smaller becomes the threshold where the discretisation error starts dominating the splitting error.

\subsection*{Harmonic oscillation on the boundary}

In the second example, we consider again the heat equation but this time we have a system of differential equations on the boundary describing the harmonic oscillator:
\begin{equation}\label{eq:example2pde}
\left\{
\begin{aligned}
\partial_tw(t,x) &= c_1\partial_{xx}w(t,x) && \text{ for } t>0,\:x\in(0,1), \\
w(0,x) &= c_2\ee^{-c_3(x-\frac 12)^2} && \text{ for } x\in(0,1), \\
\dot w(t,0) &= w(t,1) && \text{ for } t\ge 0 \text{ with } w(0,0)=1, \\
\dot w(t,1) &= -c_3w(t,0) && \text{ for } t\ge 0 \text{ with } w(0,1)=-c_4
\end{aligned}
\right.
\end{equation}
with some given constants $c_i>0, i=1,2,3,4$ and $\beta=1$.
We applied the same numerical method as described above. Note that the exact boundary function is
\begin{equation*}
v(t)=\binom{\cos(\sqrt{c_3}t)-\tfrac{c_4}{\sqrt{c_3}}\sin(\sqrt{c_3}t)}{-\sqrt{c_3}\sin(\sqrt{c_3}t)-c_4\cos(\sqrt{c_3}t)}.
\end{equation*}
 Since in this example the exact solution $u(t)$ is unknown, 
 we computed a reference solution instead, by using the same numerical method but with much larger $N_t$ value (but the same $N_x$ value). This choice is justified by the fact that, at this point, we are interested in the order of the splitting method only and not in the error of the entire numerical method (including spatial and temporal discretisations). Figure \ref{fig:exm2sol} shows the reference solution at certain time levels, and in Figure \ref{fig:exm2bdry} the time evolution of the boundary values are presented. For our numerical experiments, we chose the values $c_1=0.1, c_2=9, c_3=10, c_4=0.1$.

\begin{figure}[!ht]
\centering
\rotatebox{-90}{\includegraphics[width=6cm]{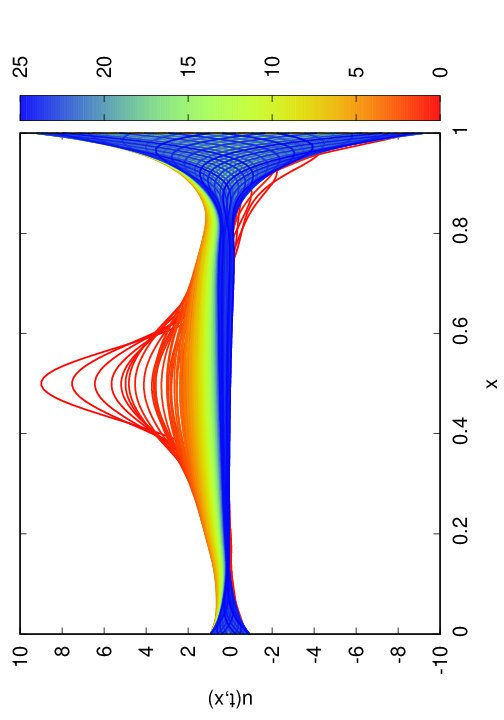}}
\caption{Exact solution $u(t)$ of Example 2 for time levels $t_n=n\cdot 0.1$, $n=1,\ldots,250$ with $N_x=128$. Red lines correspond to the beginning and blue ones to the end. \label{fig:exm2sol}}
\end{figure}

\begin{figure}[!ht]
\centering
\rotatebox{-90}{\includegraphics[width=6cm]{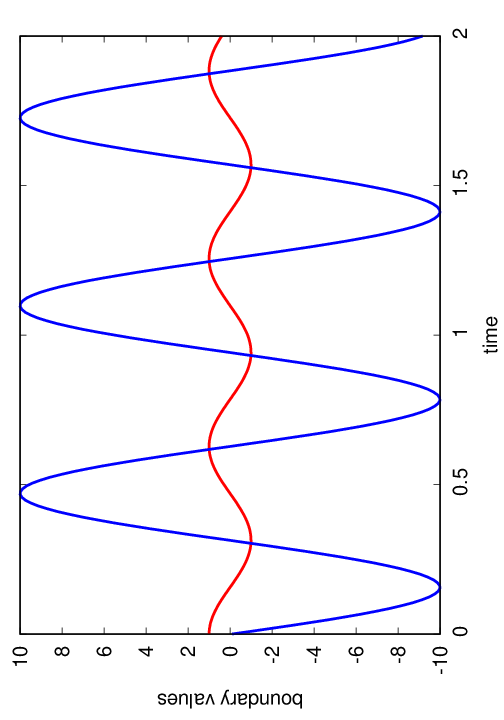}}
\caption{Exact solution $v(t)$ of Example 2 for all $t\in[0,2]$. The red line corresponds to the first coordinate function (i.e., the left boundary value), while the blue line to the second one (right boundary value). \label{fig:exm2bdry}}
\end{figure}

In Figure \ref{fig:exm2order}, the order plots of the splittings are presented for this example. As before, the slope of the lines corresponds to the approximation of the order. For comparison purposes, we included the lines with slope $1$, $1.25$ and $2$. One can see that the classical first-order splittings (Lie and weighted with $\Theta=0.3$) possess computational order $1$ as well. The symmetrically weighted splitting ($\Theta=0.5$) shows some oscillation after a while. The Strang splitting, however, behaves very well. For time step values with $\log(\tau)$ greater than approximately $-2.5$, its computational order is around $1.25$ as expected. For smaller time steps its order becomes $2$, which corresponds to the general order of the Strang splitting without the order reduction caused by the inhomogeneous boundary (the effect of the Dirichlet operator $D_0$).

\begin{figure}[!ht]
\centering
\rotatebox{-90}{\includegraphics[width=6cm]{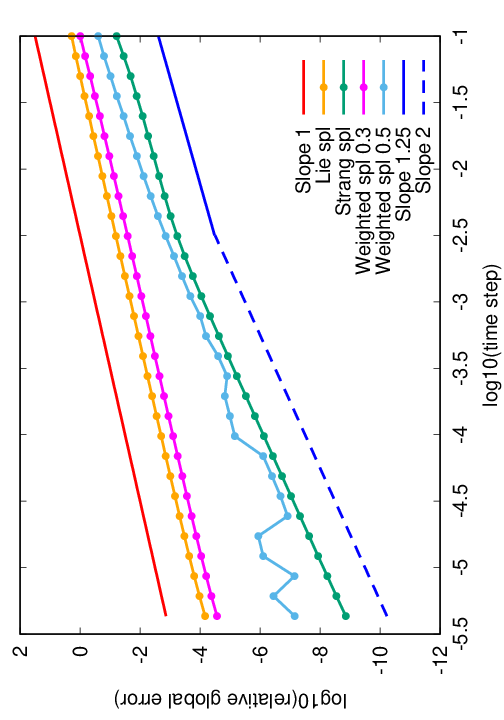}}
\caption{Order plots of the various splitting procedures for problem \eqref{eq:example1pde} with $N_x=128$. \label{fig:exm2order}}
\end{figure}

\begin{table}[!ht]
\centering
\caption{Computational orders of the splitting procedures computed as the slope of the fitted line for $N_x=128$ in the case of problem \eqref{eq:example2pde}. 
\label{tab:exm2order}}
{\footnotesize\tabcolsep0.5em
\begin{tabular}{l|c|c|c|c}
splitting & Lie & Strang & Weighted $\Theta=0.3$ & Weighted $\Theta=0.5$ \\\hline
analytical order& $\sim1$ & $\sim1.25$ & $\sim1$& $\sim1.25$ \\
computational order & $1.0100$ & $1.3056$ if $\log(\tau)>-2.5$ & $1.0256$ & $1.5765$ \\
&& $1.9812$ if $\log(\tau)<-2.5$&&
\end{tabular}}
\end{table}
By considering larger $N_x$ values (i.e., finer spatial resolution), the threshold where the break occurs can be pushed down and the numerical order gets closer to 1.25. Moreover, the oscillation for the weighted splittings in this case occurs also only for smaller $\tau$ values.

\section*{Acknowledgements}
\noindent
 The authors are indebted to Jussi Behrndt for suggesting Example \ref{examp:Lip} as an improvement to Example \ref{examp:beltrami} and providing the necessary references. 
The research was partially supported by the bilateral German-Hungarian Project \textit{CSITI -- Coupled Systems and Innovative Time Integrators}
financed by DAAD and Tempus Public Foundation. P.Cs.~acknowledges the Bolyai J\'anos Research Scholarship of the Hungarian Academy of Sciences and the support of the \'UNKP-19-4 New National Excellence Program of the Ministry of Human Capacities.  This article is based upon work from COST Action CA18232 MAT-DYN-NET, supported by COST (European Cooperation in Science and Technology).


\begin{thebibliography}{10}\normalsize

\bibitem{Altmann19}
R.~Altmann, \emph{A {PDAE} formulation of parabolic problems with dynamic
  boundary conditions}, Appl. Math. Lett. \textbf{90} (2019), 202--208.

\bibitem{Atkin09}
K.~Atkinson and W.~Han, \emph{{Theoretical Numerical Analysis: A Functional
  Analysis Framework}}, Texts in Applied Mathematics, vol.~39, Springer, 2009.

\bibitem{BagGod57}
K.~A. Bagrinovskii and S.~K. Godunov, \emph{Difference schemes for
  multidimensional problems}, Dokl. Akad. Nauk. USSR \textbf{115} (1957),
  431--433 (Russian).

\bibitem{BCsEF}
A.~B\'atkai, P.~Csom\'os, K.-J. Engel, and B.~Farkas, \emph{Operator splitting
  for operator matrices}, Int. Eq. Oper. Theor. \textbf{74} (2012), 281--299.

\bibitem{BCsF13}
A.~B\'atkai, P.~Csom\'os, and B.~Farkas, \emph{Operator splitting for
  nonautonomous delay equations}, Comput. Math. Appl. \textbf{65} (2013),
  315--324.

\bibitem{BCsF17}
A.~B\'atkai, P.~Csom\'os, and B.~Farkas, \emph{Operator splitting for
  dissipative delay equations}, Semigroup Forum \textbf{95} (2017), 345--365.

\bibitem{BCsFN11}
A.~B\'atkai, P.~Csom\'os, B.~Farkas, and G.~Nickel, \emph{Operator splitting
  for non-autonomous evolution equations}, J. Funct. Anal. \textbf{260} (2011),
  2163--2190.

\bibitem{BeGeMi20}
J.~Behrndt, F.~Gesztesy, and M.~Mitrea, \emph{Sharp boundary trace theory and
  {S}chr\"odinger operators on bounded {L}ipschitz domains},  (2020), 143
  pages, preprint.

\bibitem{Bjorhus98}
M.~Bj{\o}rhus, \emph{Operator splitting for abstract {C}auchy problems}, IMA J.
  Numer. Anal. \textbf{18} (1998), 419--443.

\bibitem{spl-magn-Schrod}
M.~Caliari, A.~Ostermann, and C.~Piazzola, \emph{A splitting approach for the
  magnetic {S}chr\"{o}dinger equation}, J. Comput. Appl. Math. \textbf{316}
  (2017), 74--85.

\bibitem{CENN}
V.~Casarino, K.-J. Engel, R.~Nagel, and G.~Nickel, \emph{A semigroup approach
  to boundary feedback systems}, Int. Eq. Op. Th. \textbf{47} (2003), no.~3,
  289--306.

\bibitem{Csomos-Nickel}
P.~Csom\'os and G.~Nickel, \emph{Operator splitting for delay equations},
  Comput. Math. Appl. \textbf{55} (2008), 2234--2246.

\bibitem{ISEM2012}
P.~Csomós, A.~B\'atkai, B.~Farkas, and A.~Ostermann, \emph{Operator semigroups
  fur numerical analysis}, Lecture notes, TULKA Internetseminar,
  https://www.math.tecnico.ulisboa.pt/ $\tilde{\
  }$czaja/ISEM/15internetseminar201112.pdf, 2012, p.~182 pages.

\bibitem{CFH05}
P.~Csomós, I.~Faragó, and A.~Havasi, \emph{Weighted sequential splitting and
  their analysis}, Comput. Math. Appl. \textbf{50} (2005), no.~7, 1017--1031.

\bibitem{Dimov-etal08}
I.~Dimov, I.~Faragó, A.~Havasi, and Z.~Zlatev, \emph{Different splitting
  techniques with application to air pollution models}, Int. J. Environment.
  Pollution \textbf{32} (2008), 174--199.

\bibitem{ADIspl-Maxw}
J.~Eilinghoff and R.~Schnaubelt, \emph{Error analysis of an {ADI} splitting
  scheme for the inhomogeneous {M}axwell equations}, Discrete Contin. Dyn.
  Syst. \textbf{38} (2018), no.~11, 5685--5709.

\bibitem{EngelMatrix}
K.-J. Engel, \emph{Matrix representation of linear operators on product
  spaces}, no.~56, 1998, International Workshop on Operator Theory (Cefal\`u,
  1997), pp.~219--224.

\bibitem{EngelOSC}
K.-J. Engel, \emph{Spectral theory and generator property for one-sided coupled
  operator matrices}, Semigroup Forum \textbf{58} (1999), no.~2, 267--295.

\bibitem{Engel-Nagel}
K.-J. Engel and R.~Nagel, \emph{One-parameter semigroups for linear evolution
  equations}, Graduate Texts in Mathematics, vol. 194, Springer-Verlag, New
  York, 2000.

\bibitem{Epshteyn20}
Y.~Epshteyn and Q.~Xia, \emph{Difference potentials method for models with
  dynamic boundary conditions and bulk-surface problems}, Adv. Comput. Math.
  \textbf{46} (2020).

\bibitem{Fukao17}
T.~Fukao, S.~Yoshikawa, and S.~Wada, \emph{Structure-preserving finite
  difference schemes for the {C}ahn-{H}illiard equation with dynamic boundary
  conditions in the one-dimensional case}, Commun. Pure Appl. Anal. \textbf{16}
  (2017), 1915--1938.

\bibitem{Geiser2011}
J.~Geiser, \emph{{Iterative Splitting Methods for Differential Equations}},
  Chapman and Hall/CRC Numerical Anal. and Sci. Comp. Series, CRC Press,
  Hoboken, NJ, 2011.

\bibitem{GeMiMiMi}
F.~Gesztesy, I.~Mitrea, D.~Mitrea, and M.~Mitrea, \emph{On the nature of the
  {L}aplace-{B}eltrami operator on {L}ipschitz manifolds}, vol. 172, 2011,
  Problems in mathematical analysis. No. 52, pp.~279--346.

\bibitem{GeMit08}
F.~Gesztesy and M.~Mitrea, \emph{Generalized {R}obin boundary conditions,
  {R}obin-to-{D}irichlet maps, and {K}rein-type resolvent formulas for
  {S}chr\"{o}dinger operators on bounded {L}ipschitz domains}, Perspectives in
  partial differential equations, harmonic analysis and applications, Proc.
  Sympos. Pure Math., vol.~79, Amer. Math. Soc., Providence, RI, 2008,
  pp.~105--173.

\bibitem{GeMit11}
F.~Gesztesy and M.~Mitrea, \emph{A description of all self-adjoint extensions
  of the {L}aplacian and {K}re\u{\i}n-type resolvent formulas on non-smooth
  domains}, J. Anal. Math. \textbf{113} (2011), 53--172.

\bibitem{Greiner}
G.~Greiner, \emph{Perturbing the boundary conditions of a generator}, Houston
  J. Math. \textbf{13} (1987), no.~2, 213--229.

\bibitem{Haase}
M.~Haase, \emph{The functional calculus for sectorial operators}, Operator
  Theory: Advances and Applications, vol. 169, Birkh\"{a}user Verlag, Basel,
  2006.

\bibitem{Hansen-Ostermann08}
E.~Hansen and A.~Ostermann, \emph{Dimension splitting for evolution equations},
  Numer. Math. \textbf{108} (2008), 557--570.

\bibitem{HO-high}
E.~Hansen and A.~Ostermann, \emph{High order splitting methods for analytic
  semigroups exist}, BIT \textbf{49} (2009), no.~3, 527--542.

\bibitem{dim-spl-quasilin-par}
E.~Hansen and A.~Ostermann, \emph{Dimension splitting for quasilinear parabolic
  equations}, IMA J. Numer. Anal. \textbf{30} (2010), no.~3, 857--869.

\bibitem{HippDiss}
D.~Hipp, \emph{A unified error analysis for spatial discretizations of
  wave-type equations with applications to dynamic boundary conditions}, Ph.D.
  thesis, Karlsruher Institut f\"ur Technologie (KIT), 2017.

\bibitem{Hipp19}
D.~Hipp and B.~Kovács, \emph{{Finite element error analysis of wave equations
  with dynamic boundary conditions: {$L^2$} estimates}}, IMA J. Numer. Anal.
  \textbf{41} (2020), no.~1, 638--728.

\bibitem{ADI-Maxwell}
M.~Hochbruck, T.~Jahnke, and R.~Schnaubelt, \emph{Convergence of an {ADI}
  splitting for {M}axwell's equations}, Numer. Math. \textbf{129} (2015),
  no.~3, 535--561.

\bibitem{HLR13}
H.~Holden, C.~Lubich, and N.~H. Risebro, \emph{Operator splitting for partial
  differential equations with {B}urgers nonlinearity}, Math. Comp. \textbf{82}
  (2013), 173--185.

\bibitem{HunVer03}
W.~Hundsdorfer and J.~G. Verwer, \emph{{Solution of Time-dependent
  Advection-Diffusion-Reaction Equations}}, Springer Series in Computational
  Mathematics, vol.~33, Springer, 2003.

\bibitem{JahnkeLubich}
T.~Jahnke and C.~Lubich, \emph{Error bounds for exponential operator
  splittings}, BIT \textbf{40} (2000), no.~4, 735--744.

\bibitem{spl-schrod-damp}
T.~Jahnke, M.~Mikl, and R.~Schnaubelt, \emph{Strang splitting for a semilinear
  {S}chr\"{o}dinger equation with damping and forcing}, J. Math. Anal. Appl.
  \textbf{455} (2017), no.~2, 1051--1071.

\bibitem{Jacobsen-etal01}
E.~R. Jakobsen, K.~Hvistendahl~Karlsen, and N.~H. Risebro, \emph{On the
  convergence rate of operator splitting for {H}amilton-{J}acobi equations with
  source terms}, SIAM J. Numer. Anal. \textbf{39} (2001), no.~2, 499--518.

\bibitem{Knopf19}
P.~Knopf and K.~F. Lam, \emph{Convergence of a {R}obin boundary approximation
  for a {C}ahn{\textendash}{H}illiard system with dynamic boundary conditions},
  Nonlinearity \textbf{33} (2020), no.~8, 4191--4235.

\bibitem{Knopf20}
P.~Knopf, K.~F. Lam, C.~Liu, and S.~Metzger, \emph{Phase-field dynamics with
  transfer of materials: {T}he {C}ahn--{H}illard equation with reaction rate
  dependent dynamic boundary conditions}, ESAIM: M2AN \textbf{55} (2021),
  no.~1, 229--282.

\bibitem{KnoSig20}
P.~Knopf and A.~Signori, \emph{On the nonlocal {C}ahn--{H}illiard equation with
  nonlocal dynamic boundary condition and boundary penalization}, J. Diff. Eq.
  \textbf{280} (2021), 236--291.

\bibitem{KLL17}
B.~Kovács, B.~Li, and C.~Lubich, \emph{Convergence of finite elements on an
  evolving surface driven by diffusion on the surface}, Numer. Math.
  \textbf{137} (2017), 643--689.

\bibitem{Kov17}
B.~Kovács and C.~Lubich, \emph{Numerical analysis of parabolic problems with
  dynamic boundary conditions}, IMA J. Numer. Anal. \textbf{37} (2017), 1--39.

\bibitem{Langa19}
F.~Langa and M.~Pierre, \emph{A doubly splitting scheme for the caginalp system
  with singular potential and dynamic boundary conditions}, HAL preprint
  hal-02310210 (2019).

\bibitem{Lax56}
P.~D. Lax and R.~D. Richtmyer, \emph{Survey of the stability of linear finite
  difference equations}, CPAM \textbf{9} (1956), 267--293.

\bibitem{LMa}
J.-L. Lions and E.~Magenes, \emph{Non-homogeneous boundary value problems and
  applications. {V}ol. {I}}, Springer-Verlag, New York-Heidelberg, 1972,
  Translated from the French by P. Kenneth, Die Grundlehren der mathematischen
  Wissenschaften, Band 181.

\bibitem{Lunardi}
A.~Lunardi, \emph{Analytic semigroups and optimal regularity in parabolic
  problems}, Modern Birkh\"{a}user Classics, Birkh\"{a}user/Springer Basel AG,
  Basel, 1995, (2013 reprint of the 1995 original).

\bibitem{Marchuk68}
G.~I. Marchuk, \emph{Some application of splitting-up methods to the solution
  of mathematical physics problems}, Applik. Mat. \textbf{13} (1968), no.~2,
  103--132.

\bibitem{McL00}
W.~McLean, \emph{Strongly elliptic systems and boundary integral equations},
  Cambridge University Press, Cambridge, 2000.

\bibitem{Mugnolo}
D.~Mugnolo, \emph{A note on abstract initial boundary value problems},
  T\"ubinger Berichte zur Funktionalanalysis \textbf{10} (2001), 158--162.

\bibitem{MugnoloDiss}
D.~Mugnolo, \emph{Second order abstract initial-boundary value problems}, Ph.D.
  thesis, Universit\"at T\"ubingen, 2004.

\bibitem{Sportisse00}
B.~Sportisse, \emph{An analysis of operator splitting techniques in the stiff
  case}, J. Comput. Phys. \textbf{161} (2000), 140--168.

\bibitem{Strang68}
G.~Strang, \emph{On the construction and comparison of difference schemes},
  SIAM J. Numer. Anal. \textbf{5} (1968), no.~3, 506--517.

\bibitem{Trotter59}
H.~F. Trotter, \emph{On the product of semi-groups of operators}, Proc. Amer.
  Math. Soc. \textbf{10} (1959), 545--551.

\end{thebibliography}

\end{document}